\documentclass[11pt]{article}

\usepackage{amsmath,amssymb,amsthm,amsfonts}
\usepackage{mathtools}
\usepackage{mathrsfs}
\usepackage{bm}
\usepackage{graphicx}
\usepackage{float}
\usepackage{multirow}
\usepackage{booktabs}
\usepackage{longtable}
\usepackage{arydshln}
\usepackage{enumerate}
\usepackage{cite}

\usepackage{hyperref}
\usepackage[a4paper,margin=2.5cm]{geometry}
\hypersetup{colorlinks=true,linkcolor=blue,citecolor=blue,urlcolor=blue,
  pdftitle={Beurling-Type Theorem in Weighted Bergman Spaces},
  pdfauthor={Zhaopeng Lin, Shibo Xu, Tao Yu}}
\allowdisplaybreaks[2]
\numberwithin{equation}{section}
\newtheorem{theorem}{Theorem}[section]
\newtheorem{lemma}[theorem]{Lemma}
\newtheorem{proposition}[theorem]{Proposition}
\newtheorem{corollary}[theorem]{Corollary}

\newtheorem{remark}[theorem]{Remark}

\newcommand{\D}{\mathbb D}
\newcommand{\R}{\mathbb R}
\newcommand{\C}{\mathbb C}
\newcommand{\Z}{\mathbb Z}
\newcommand{\Hh}{\mathcal H_\alpha}
\newcommand{\Mh}{\mathcal M_h}

\DeclareMathOperator{\spanop}{span}
\DeclareMathOperator{\Ran}{Ran}
\DeclareMathOperator{\Hol}{Hol}
\DeclareMathOperator{\Res}{Res}
\DeclareMathOperator{\sech}{sech}

\title{\Large\bfseries The Critical Value for the Beurling-Type Theorem in Weighted Bergman Spaces
\thanks{This work was supported by the National Natural Science Foundation
of China (Grant No.~ 12571132).}}

\author{Zhaopeng Lin, Shibo Xu\thanks{Corresponding author.}, and Tao Yu}

\date{}
\begin{document}
\maketitle
\vspace{-0.8cm}
\begin{center}
\begin{minipage}{14cm}\small
\noindent\textbf{Abstract}\quad
For every $\alpha>1$ we construct a finite set $A\subset\D$ such that the
zero-based invariant subspace $I_A$ fails the wandering subspace property.  Consequently, the Beurling-type theorem holds on the weighted Bergman space
\(A^2_\alpha\) if and only if \(-1<\alpha\le1\), confirming a conjecture of Shimorin. As a further application,  we extend the prescribed-curvature construction of
Hedenmalm and Perdomo to all $\alpha>1$, thereby replacing the constant
$\alpha_0\approx1.04$ in their result by $1$.

\medskip
\noindent\textbf{Mathematics Subject Classification (2020).}\quad
Primary 47A15; Secondary 30H20, 46E22.

\medskip
\noindent\textbf{Keywords.}\quad
Weighted Bergman space, wandering subspace, zero-based invariant
subspace,  extraneous zero.
\end{minipage}
\end{center}

\section{Introduction}\label{sec:introduction}

Let $T$ be a bounded operator on a Hilbert space $\mathcal H$, and let
$\mathcal M$ be a closed $T$-invariant subspace. Its wandering subspace is
\[
   \mathcal M\ominus T\mathcal M.
\]
For $\mathcal E\subset\mathcal H$, let $ [\mathcal E]_T$ denote the $T$-invariant subspace generated by $\mathcal E$, i.e.
\[
   [\mathcal E]_T
   =
   \overline{\operatorname{span}}
   \{T^n f:n\ge0,\ f\in\mathcal E\}.
\]
We say that $\mathcal M$ has the wandering subspace property if
\begin{equation}\label{wsp}
   \mathcal M=[\mathcal M\ominus T\mathcal M]_T.
\end{equation}
The Beurling-type theorem is said to hold for $T$ on $\mathcal H$ if
\eqref{wsp} holds for every closed $T$-invariant subspace $\mathcal M$.

Shift operators  appear naturally in the Wold--Kolmogorov decomposition. For a pure isometry, the operator is
unitarily equivalent to a unilateral shift. Shimorin showed that
related Wold-type questions for a broader class of operators can also
be studied by means of shift operators acting on Hilbert spaces of
analytic functions; see \cite{Shimorin2001}. This leads naturally to
the Beurling-type theorem for the shift operator
\[
 T_zf(z)=zf(z)
\]
on Hilbert spaces of analytic functions.

Let \(\D=\{z\in\C:|z|<1\}\). For the Hardy space \(H^2(\D)\), the classical Beurling theorem
\cite{B} states that every nonzero closed \(T_z\)-invariant subspace
\(\mathcal M\) of \(H^2(\D)\) is of the form $
 \mathcal M=\theta H^2(\D)$ 
for some inner function \(\theta\), and
\[
 \mathcal M\ominus z\mathcal M=\theta \C.
\]
Hence the Beurling-type theorem holds for \(T_z\) on \(H^2(\D)\).

Richter \cite{RS} established the Beurling-type theorem
for the classical Dirichlet space. This result also
follows from the general wandering subspace theorem
for analytic concave operators, since the Dirichlet
shift is a $2$-isometry.

 For $\alpha>-1$, the weighted Bergman space $A^2_\alpha(\D)$ consists of the
holomorphic functions $f$ on $\D$ such that
\begin{equation}\label{eq:bergman-norm}
 \|f\|_\alpha^2=\int_\D |f(z)|^2\,dA_\alpha(z)<\infty,
\end{equation}
where $ dA_\alpha(z)=(\alpha+1)(1-|z|^2)^\alpha dA(z)$ and  $
dA(z)=dx\,dy/\pi$.

In contrast to the Hardy shift, \(T_z\) is not an isometry on
\(A^2_\alpha\), and hence the classical isometric structure underlying
Beurling's theorem is no longer available. It is therefore natural to
ask for which \(\alpha>-1\) the Beurling-type theorem holds for \(T_z\)
on \(A^2_\alpha\).

 Aleman, Richter and Sundberg \cite{ARS}
proved that the Beurling-type theorem holds for the unweighted
Bergman space $A^2_0$. Sun and Zheng \cite{SZ} later reproved this result using the lifting of the Bergman shift to the Hardy space of the bidisk. Shimorin subsequently proved \cite{Shimorin}
that it holds throughout the range $
   -1<\alpha\le1$. McCullough and Richter \cite{MR} established wandering
subspace and dilation theorems for shifts on Bergman-type
reproducing kernel Hilbert spaces. Olofsson \cite{OA}
extended wandering subspace theorems to a broader class
of pure expansive operators satisfying suitable growth
conditions.

To investigate the failure of the Beurling-type theorem, it is natural
to consider invariant subspaces determined by prescribed zero sets. 
For any set \(A\subset\D\), define
\begin{equation}\label{eq:zero-space}
I_A=\{f\in A^2_\alpha:f(a)=0\text{ for every }a\in A\}.
\end{equation}
Since point evaluations are bounded on $A^2_\alpha$, the space
$I_A$ is closed; moreover, it is invariant under $T_z$.
Such subspaces are called zero-based invariant subspaces
(see \cite{HP,Shimorin2001}).

Let $K_A$ denote the reproducing kernel of $I_A$.  If \(A\subset\D\setminus\{0\}\) is finite, then the wandering subspace is
one-dimensional; more precisely, as shown in
Section~\ref{sec:preliminaries},
\[
 I_A\ominus zI_A=\operatorname{span}\{K_A(\cdot,0)\}.
\]
A zero of $K_A(\cdot,0)$ outside the prescribed zero set $A$ is
called an extraneous zero.  The existence of such a zero implies
\[
   I_A\ne[I_A\ominus zI_A]_{T_z}.
\]

The study of such zero-based invariant subspaces is closely related
to extremal functions and canonical divisors in Bergman spaces.
Hedenmalm \cite{Hedenmalm1991} introduced canonical divisors through
an extremal problem for prescribed zero sets, and this point of view
was further developed for invariant subspaces in
\cite{DKSS1993,DKS1996}.  In the present Hilbert-space setting, the normalized
extremal function associated with \(I_A\) is a positive multiple of
\(K_A(\cdot,0)\).  Thus extraneous zeros of the extremal function provide a direct
obstruction to the wandering subspace property.

In 1992,  Hedenmalm and Zhu \cite{HZ} showed that the factorization properties available in the range 
\(-1<\alpha\le1\) break down once \(\alpha>1\).  They also showed
that for \(\alpha>4\) there exist zero-based invariant subspaces whose
extremal functions have extraneous zeros, and hence the
Beurling-type theorem fails in that range. 

Hedenmalm and Perdomo \cite{HP} later narrowed this gap
substantially. By relating extraneous zeros of weighted Bergman kernels
to the optimization problem for mean value surfaces with prescribed
curvature form, they showed that the Beurling-type theorem fails for
\[
\alpha>\alpha_0,\qquad \alpha_0\approx1.04.
\]
Their construction also suggested that the threshold
\(\alpha_0\) should be replaced by \(1\).

Hedenmalm and Perdomo \cite[Section~4]{HP}
also suggested an analytical approach in the
upper half-plane, replacing the discrete zeros
by a continuous distribution along two rays
symmetric about the imaginary axis.
They expected that suitable choices of the angle
and density might produce a zero of the
reproducing kernel associated with the resulting
weight for every $\alpha>1$, but noted that
computing this kernel is difficult in the
absence of strong symmetry.

Through study of wandering subspaces in weighted
Bergman spaces, Shimorin \cite{Shimorin2001} proposed the following conjecture:

Shimorin's Conjecture: The Beurling-type theorem fails for weighted Bergman spaces $A_{\alpha}^2$ when $\alpha>1$.

It remained open whether the Beurling-type theorem holds in the range 
\[
 1<\alpha\leq\alpha_0\approx1.04.
\]

Wu and Yu \cite{WY} proved that every one-point
zero-based invariant subspace of $A^2_2$ has the
wandering subspace property by lifting the shift
operator to a quotient module of the Hardy space
on the bidisk.
Wu, Wang and Yu \cite{WWY} subsequently studied
the wandering subspace property for such subspaces
of $A^2_0$ and $A^2_2$ through the spectra of
Toeplitz operators on the Hardy space $H^2$.

More recently, Gu and Park \cite{GuPark}
studied finite zero-based invariant subspaces
in reproducing kernel Hilbert spaces of
analytic functions on $\D$, using
Toeplitz-operator methods to characterize
the wandering subspace property.
 As an application to the weighted Bergman spaces, they proved
that every one-point zero-based invariant subspace $I_a$ has the
wandering subspace property when $1<\alpha\le4$.

In this paper, we answer Shimorin's conjecture affirmatively by
constructing a class of finite zero-based invariant subspaces.

\begin{theorem}\label{thm:main}
For every \(\alpha>1\), there exists a finite set
\(A \subset\D\setminus\{0\}\) such that \(I_A\) fails the wandering subspace property.
More precisely, there exists \(c\in\D\setminus A\) such that $ K_A(c,0)=0$.
\end{theorem}

The positive result of Shimorin \cite{Shimorin}, together with
Theorem~\ref{thm:main}, gives the sharp range for the Beurling-type
theorem.

\begin{corollary}\label{cor:critical-range}
For \(\alpha>-1\), the Beurling-type theorem holds for \(T_z\) on
\(A^2_\alpha\) if and only if \(\alpha\le1\).
\end{corollary}

Following the notation of Hedenmalm and Perdomo \cite{HP}, let
\[
\Delta=\frac14(\partial_x^2+\partial_y^2),
\]
and denote by
\[
\boldsymbol K_{\mathbb H}(z)
=
-\frac{4}{(1-|z|^2)^2}dA(z)
\]
the curvature form of the Poincaré metric. We recall that the optimization problem \textup{(OP)} of Hedenmalm and
Perdomo \cite{HP} asks, for a prescribed curvature density \(\mu\), to
maximize \(\omega(0)\) among positive smooth weights satisfying
\[
\Delta\log\omega=-\frac12\mu,
\qquad
\int_{\mathbb D}\omega\,dA=1.
\]
The existence of a smooth positive optimizer is equivalent to a
zero-free condition for the associated weighted Bergman kernel.
The stronger conclusion in Theorem~\ref{thm:main},  the existence
of an extraneous zero of the reproducing kernel, also has a consequence
for the prescribed curvature problem.
\begin{corollary}\label{cor:HP}
For every $\alpha>1$, there exists a real-valued
$\mu\in C^\infty(\overline{\mathbb D})$ such that
\[
\boldsymbol{\mu}
+\frac{\alpha}{2}\boldsymbol K_{\mathbb H}\le0 ,
\]
where $
\boldsymbol{\mu}=\mu\,dA$, but the associated optimization
problem \textup{(OP)} admits no $C^\infty$-smooth positive solution.
\end{corollary}

Inspired by  the two-ray configuration
of Hedenmalm and Perdomo \cite{HP}, we transfer the problem to a weighted Bergman space
on a horizontal strip and construct a periodic zero
set on two horizontal lines. Furthermore, using a weighted Paley–Wiener-type theorem \cite{Harper}, we identify this space isomorphically with a Hilbert space consisting of sequences of functions.   Using the Fourier representation, we   give a precise characterization of the corresponding zero-based invariant subspaces in terms of two linear constraints.
Next, by solving the corresponding minimum norm problem, we derive an integral representation of the limit of  reproducing kernels.  

We analyze the asymptotic behavior of the limit function on the real axis and prove the existence of a zero-based invariant subspace whose reproducing kernel at the origin has extraneous zeros. Finally, by truncating
the periodic zero set, we obtain a finite zero set \(A\) such
that \(K_A(\cdot,0)\) has an extraneous zero.

The paper is organized as follows.
Section~\ref{sec:preliminaries} collects the necessary preliminaries.
Section~\ref{sec:lattice} introduces the periodic zero configuration
and studies the associated reproducing kernels.
Section~\ref{sec:negative} establishes an asymptotic formula for the limit function of the reproducing kernel.  Section~\ref{sec:finite} completes the proof of
Theorem~\ref{thm:main} by constructing a finite zero set.

\section{Preliminaries }\label{sec:preliminaries} 
 
\subsection{Zero-based subspaces in weighted Bergman spaces}

Note that on the Bergman space $A^2_\alpha$ the monomials are mutually orthogonal, and 
\begin{equation}\label{eq:monomial-norm}
 \|z^n\|_\alpha^2
 =\frac{\Gamma(\alpha+2)n!}{\Gamma(n+\alpha+2)},
 \qquad n\ge0.
\end{equation}
It follows that  the reproducing kernel of $A^2_\alpha$ is
\begin{equation}\label{eq:disk-kernel}
 K_\alpha(z,w)
 =\frac{1}{(1-z\overline w)^{\alpha+2}},
 \qquad z,w\in\D.
\end{equation}
If $M$ is a closed subspace of $A^2_\alpha$, then $M$ is also a reproducing kernel Hilbert space.
For background on Bergman spaces, we refer to \cite{HKZ}.

The following result is presumably known, but we have not been able to locate a precise reference; we therefore provide a proof for the sake of completeness.

\begin{lemma}\label{lem:wandering}\label{lem:extra}
Let \(A\subset\D\setminus\{0\}\) and suppose that \(I_A\neq\{0\}\), where \(I_A\) is defined by \eqref{eq:zero-space}.
Then the zero set of \(I_A\) is equal to $A$,
\[
zI_A=\{f\in I_A:f(0)=0\},
\qquad
K_A(0,0)>0,
\]
and
\begin{equation}\label{eq:wandering-kernel}
I_A\ominus zI_A
=
\operatorname{span}\{K_A(\cdot,0)\},
\end{equation}
where \(K_A\) denotes the reproducing kernel of 
\(I_A\).
Moreover, if \(K_A(c,0)=0\) for some \(c\in\D\setminus A\), then
\[
I_A\ne[I_A\ominus zI_A]_{T_z}.
\]
\end{lemma}

\begin{proof}
Suppose $b\in\mathbb{D}\setminus A$. Choose a nonzero function \( f\in I_A\), and let \(m\ge0\) be the order of the zero
of \(f\) at \(b\). Since $z-b$ is a divisor of \(A^2_\alpha\) and $b\notin A$, $(z-b)^{-m}f$ belongs to $I_A$. The point $b$ is therefore not a zero of $I_A$. Consequently, the zero set of \(I_A\) is equal to $A$. 

Since $0\notin A$, $z$ is a divisor of $I_A$. It follows that
\[
zI_A=\{f\in I_A:f(0)=0\}=I_{A\cup\{0\}}.
\]
Let $f_1, f_2 \in I_A\setminus zI_A$. Then there exists a constant \(c \in \mathbb{C}\) such that \(f_1(0)+c f_2(0)=0\), which implies  \(f_1+c f_2 \in z I_A\).
 Therefore $\dim(I_A\ominus zI_A)=1$.

Recalling that  the zero set of \(I_A\) is equal to $A$, we have that the evaluation functional at zero on \(I_A\) is not zero, and so the kernel \(K_A(\cdot,0)\neq 0\). Consequently,
\[
K_A(0,0)
=
\|K_A(\cdot,0)\|_\alpha^2
>0.
\]
 
The reproducing property now gives the identity
\eqref{eq:wandering-kernel}.

Finally, suppose that \(K_A(c,0)=0\) and that \(c\) is not a common
zero of \(I_A\). Since
\[
I_A\ominus zI_A
=
\spanop\{K_A(\cdot,0)\},
\]
every polynomial multiple of \(K_A(\cdot,0)\) vanishes at \(c\).
As evaluation at \(c\) is bounded on \(A^2_\alpha\), it follows that
$h(c)=0$ for all $h\in[I_A\ominus zI_A]_{T_z}$.
On the other hand, there
exists \(f\in I_A\) such that \(f(c)\neq0\). Therefore
\[
I_A\neq[I_A\ominus zI_A]_{T_z}.
\]
\end{proof}

\subsection{Weighted Bergman space on the strip}

Let
\[
 S=\left\{u=x+it:|t|<\frac{\pi}{2}\right\},
\]
and define
\[
 \psi(u)=\tanh(u/2).
\]
 Then $\psi$ is a conformal bijection from $S$ onto $\D$.

We shall study in the   weighted Bergman space \(\mathcal H_\alpha\) on \(S\)
\[
 \Hh=
 \left\{F\in\Hol(S):
 \|F\|_{\Hh}^2
 =
 \int_\R\int_{-\pi/2}^{\pi/2}
 |F(x+it)|^2\cos^\alpha t\,dt\,dx<\infty
 \right\}.
\]

\begin{proposition}\label{lem:unitary}\label{lem:kernel-unitary}\label{lem:real}
Let
\[
J_\alpha(u)
=
\sqrt{\frac{\alpha+1}{4\pi}}\,
\cosh(u/2)^{-\alpha-2}.
\]
 Then the following statements hold.
\begin{enumerate}[(i)]
\item The map
\begin{equation}\label{eq:unitary-map}
(Uf)(u)=J_\alpha(u)f(\psi(u)).
\end{equation}
is unitary from $A^2_\alpha$ onto $\Hh$. The function $J_\alpha$ is
zero-free on $S$ and positive on $\R$.
\item If $M\subset A^2_\alpha$ is a closed subspace and $\widetilde M=UM$, then
\begin{equation}\label{eq:kernel-transform}
 K_{\widetilde M}(u,v)
 =J_\alpha(u)\overline{J_\alpha(v)}
   K_M(\psi(u),\psi(v)),
\end{equation}
where $ K_{\widetilde M}$ and $K_M$ denote the reproducing kernels of $\widetilde M$ and $M$.
Moreover, $Uf$ vanishes on $\Lambda\subset S$ if and only if $f$ vanishes
on $\psi(\Lambda)$.
\item If $\Lambda\subset S$ is invariant under complex conjugation and $
 \mathcal M=\{F\in\Hh:F|_\Lambda=0\}$, 
then $K_{\mathcal M}(x,0)\in\R$ for every $x\in\R$.  
If $A\subset\D$ is invariant under complex
conjugation, then
\[
K_{I_A}(x,0)\in\R,
\qquad x\in(-1,1).
\]
 Consequently, if \(A=\psi(\Lambda)\), then
\[
\operatorname{sgn} K_{\mathcal M}(x,0)
=
\operatorname{sgn} K_{I_A}(\psi(x),0),
\qquad x\in\mathbb R.
\]
\end{enumerate}
\end{proposition}
\begin{proof}
We begin with the elementary identities.
For $u=x+it$,
\[
 \cosh\frac{u}{2}
 =\cosh\frac{x}{2}\cos\frac{t}{2}
   +i\sinh\frac{x}{2}\sin\frac{t}{2},
\]
and a direct computation gives
\[
 \left|\cosh\frac{u}{2}\right|^2
 =\frac{\cosh x+\cos t}{2},
 \qquad
 \left|\sinh\frac{u}{2}\right|^2
 =\frac{\cosh x-\cos t}{2}.
\]
Since $\psi(u)=\sinh(u/2)/\cosh(u/2)$,
\[
 1-|\psi(u)|^2
 =\frac{|\cosh(u/2)|^2-|\sinh(u/2)|^2}
        {|\cosh(u/2)|^2}
 =\frac{\cos t}{|\cosh(u/2)|^2},
\]
and
\[
 \psi'(u)=\frac1{2\cosh^2(u/2)}.
\]
Hence
\begin{equation}\label{eq:strip-jacobian}
 1-|\psi(u)|^2
 =\frac{\cos t}{|\cosh(u/2)|^2},
 \qquad
 |\psi'(u)|^2
 =\frac1{4|\cosh(u/2)|^4}.
\end{equation}

We now prove (i).  By a change of variables
$z=\psi(u)$, we obtain
\begin{align*}
 \|f\|_\alpha^2
 &=(\alpha+1)
   \int_\D|f(z)|^2
   (1-|z|^2)^\alpha dA(z)\\
 &=\frac{\alpha+1}{\pi}
   \int_S|f(\psi(u))|^2
   (1-|\psi(u)|^2)^\alpha|\psi'(u)|^2\,dx\,dt\\
 &=\frac{\alpha+1}{4\pi}
   \int_S|f(\psi(u))|^2
   |\cosh(u/2)|^{-2\alpha-4}
   \cos^\alpha t\,dx\,dt\\
 &=\int_S|J_\alpha(u)f(\psi(u))|^2\cos^\alpha t\,dx\,dt\\
 &=\|Uf\|_{\Hh}^2.
\end{align*}
Thus $U$ is an isometry. Since \(\psi(u)=\tanh(u/2)\) is conformal on \(S\),
\(\cosh(u/2)\) has no zeros in \(S\).
 As \(S\) is simply connected, one can choose a holomorphic branch of
\(\log\cosh(u/2)\) on \(S\), which we take to be real on \(\R\).  
Thus \(J_\alpha\) is zero-free on \(S\), and $J_\alpha(x)>0$ for all $x\in\R$.  If $F\in\Hh$, then
\[
 f(z)=\left(\frac{F}{J_\alpha}\right)(\psi^{-1}(z))
\]
is holomorphic on $\D$, and 
$f\in A^2_\alpha$ and $Uf=F$.  Therefore $U$ is onto and hence unitary.

(ii) It is a classical result that unitary equivalence of reproducing kernel Hilbert spaces naturally induces a kernel transformation formula, and \eqref{eq:kernel-transform} follows from \eqref{eq:unitary-map}. 
The zero correspondence follows immediately from
\eqref{eq:unitary-map}, 
since $J_\alpha$ is zero-free.

For (iii), define the conjugation $
 (\mathcal JF)(u)=\overline{F(\overline u)}$. 
The function $\cos^\alpha t$ is even in $t$, and hence $\mathcal J$ is an
antilinear isometric involution on $\Hh$.  If $\Lambda$ is invariant under
complex conjugation, then $\mathcal J\mathcal M=\mathcal M$.  Put $
 k=K_{\mathcal M}(\cdot,0)$. 
For $F\in\mathcal M$,  a direct computation gives
\[
 \langle F,\mathcal Jk\rangle
 =\overline{\langle\mathcal JF,k\rangle}
 =\overline{(\mathcal JF)(0)}
 =F(0).
\]
Thus $\mathcal Jk$ represents the same evaluation functional as $k$.
By uniqueness, $\mathcal Jk=k$.
Hence, for real $x$,
\[
 k(x)=\overline{k(x)},
\]
it follows that $K_{\mathcal M}(x,0)$ is real. Let
$A\subset\D$ be invariant under complex 
conjugation 
Since $\psi^{-1}$ preserves complex conjugation,
$\psi^{-1}(A)$ is also conjugation-invariant.  
By the kernel transformation formula,
\[
K_{I_A}(\psi(x),0)
=
\frac{K_{UI_A}(x,0)}
{J_\alpha(x)J_\alpha(0)}
\in\R,
\qquad x\in\R.
\]
Since $\psi(\R)=(-1,1)$, the assertion
for $I_A$ follows. Furthermore, since
$J_\alpha(x)J_\alpha(0)>0$ for $x\in\R$,
the sign identity follows as well.

\end{proof}

\subsection{The weighted Lebesgue space}

Fix \(\alpha>1\). Using Harper’s weighted Paley–Wiener theorem, $\Hh$ can be identified with a weighted Lebesgue space on the real line. 
To this end, we define the Laplace transform of $\cos^{\alpha}$ as
\begin{equation}\label{eq:m-definition}
m(z)
=
\int_{-\pi/2}^{\pi/2}
e^{-2zt}\cos^\alpha t\,dt.
\end{equation}

\begin{proposition}\label{lem:growth-m}
The function $m$ is even and entire, and $m(x)>0$ for $x\in\R$.  As
$|x|\to\infty$,
\begin{equation}\label{eq:laplace}
m(x)
 =
 \frac{\Gamma(\alpha+1)}
      {(2|x|)^{\alpha+1}}
 e^{\pi |x|}
 \bigl(1+O(x^{-2})\bigr).
\end{equation}
For all real $x,y$,
\begin{equation}\label{eq:complex-bound}
 |m(x+iy)|\le m(x).
\end{equation}
In particular,   $1/m\in L^1(\R)$.
\end{proposition}
\begin{proof}  Differentiation under the integral sign shows that $m$
is entire. The substitution $s=-t$ in the integral shows that $m$ is even.
 Finally, for \(x\in\R\), by \eqref{eq:m-definition},
 $m(x) > 0$.

We next prove \eqref{eq:laplace}. Since $m$ is even,  it suffices to consider the part where $x>0$.
Put \(t=-\pi/2+s\). Then
\[
 m(x)
 =
 e^{\pi x}\int_0^\pi e^{-2xs}\sin^\alpha s\,ds.
\]
Fix $\delta\in(0,\pi/2)$, we have
\[
 \sin^\alpha s=s^\alpha\bigl(1+O(s^2)\bigr),
 \qquad 0\le s\le\delta.
\]
Hence
\begin{equation}\label{eq:laplace-local}
 \int_0^\delta e^{-2xs}\sin^\alpha s\,ds
 =
 \int_0^\delta e^{-2xs}s^\alpha\,ds
 +O\left(
   \int_0^\delta e^{-2xs}s^{\alpha+2}\,ds
 \right).
\end{equation}

For the error term, we have
\[
 \int_0^\delta e^{-2xs}s^{\alpha+2}\,ds
 \le
 \int_0^\infty e^{-2xs}s^{\alpha+2}\,ds
 =
 \frac{\Gamma(\alpha+3)}{(2x)^{\alpha+3}}=O(x^{-\alpha-3}).
\] 

We next evaluate the main term. Writing
\[
 \int_0^\delta e^{-2xs}s^\alpha\,ds
 =
 \int_0^\infty e^{-2xs}s^\alpha\,ds
 -
 \int_\delta^\infty e^{-2xs}s^\alpha\,ds,
\]
the first integral is, by the change of variables \(u=2xs\),
\[
 \int_0^\infty e^{-2xs}s^\alpha\,ds
 =
 \frac{\Gamma(\alpha+1)}{(2x)^{\alpha+1}}.
\]
For the remaining tail, writing \(s=\delta+r\), we obtain for \(x\ge1\)
\[
 \begin{aligned}
 \int_\delta^\infty e^{-2xs}s^\alpha\,ds
 &=
 e^{-2\delta x}
 \int_0^\infty e^{-2xr}(\delta+r)^\alpha\,dr\\
 &\le
 e^{-2\delta x}
 \int_0^\infty e^{-2r}(\delta+r)^\alpha\,dr\\
 &=O(e^{-2\delta x}).
 \end{aligned}
\]
Consequently,
\begin{equation}\label{eq:laplace-main}
 \int_0^\delta e^{-2xs}s^\alpha\,ds
 =
 \frac{\Gamma(\alpha+1)}{(2x)^{\alpha+1}}
 +O(e^{-2\delta x}).
\end{equation}

Finally, on \([\delta,\pi]\) we have \(0\le\sin^\alpha s\le1\), and hence
\[
 \int_\delta^\pi e^{-2xs}\sin^\alpha s\,ds
 \le
 \int_\delta^\pi e^{-2xs}\,ds
 \le
 \frac{e^{-2\delta x}}{2x}
 =O(e^{-2\delta x}).
\]
Combining this estimate with
\eqref{eq:laplace-local}--\eqref{eq:laplace-main}, we obtain
\[
 \int_0^\pi e^{-2xs}\sin^\alpha s\,ds
 =
 \frac{\Gamma(\alpha+1)}{(2x)^{\alpha+1}}
 +O(x^{-\alpha-3})
 +O(e^{-2\delta x}).
\]
Since the exponential term is \(O(x^{-\alpha-3})\), this becomes
\[
 \int_0^\pi e^{-2xs}\sin^\alpha s\,ds
 =
 \frac{\Gamma(\alpha+1)}{(2x)^{\alpha+1}}
 \bigl(1+O(x^{-2})\bigr).
\]
Recalling that
\[
 m(x)
 =
 e^{\pi x}\int_0^\pi e^{-2xs}\sin^\alpha s\,ds
\]
which proves \eqref{eq:laplace}.

For real $x,y$,
\[
 |m(x+iy)|
 \le\int_{-\pi/2}^{\pi/2}|e^{-2(x+iy)t}|\cos^\alpha t\,dt
 =\int_{-\pi/2}^{\pi/2}e^{-2xt}\cos^\alpha t\,dt
 =m(x),
\]
which proves \eqref{eq:complex-bound}.  

Finally, \eqref{eq:laplace} obviously implies
$1/m\in L^1(\R)$.
\end{proof}

\begin{corollary}\label{lem:fourier}
Every \(F\in\Hh\) admits a unique
\(\varphi\in L^2(\R,m(\xi)\,d\xi)\) such that
\begin{equation}\label{eq:fourier}
 \begin{split}
 F(x+it)&=\frac1{\sqrt{2\pi}}
    \int_\R\varphi(\xi)e^{ix\xi-t\xi}\,d\xi,\\
 \|F\|_{\Hh}^2&=\int_\R|\varphi(\xi)|^2m(\xi)\,d\xi.
 \end{split}
\end{equation}
Conversely, every
\(\varphi\in L^2(\R,m(\xi)\,d\xi)\)
defines an element of \(\Hh\) by \eqref{eq:fourier}. 
The Laplace representation also permits differentiation under the
integral sign. For every integer \(j\ge0\),
\[
F^{(j)}(z)
=
\frac{1}{\sqrt{2\pi}}
\int_{\mathbb R}
(i\xi)^j\varphi(\xi)e^{iz\xi}\,d\xi ,
\qquad z\in S.
\]
Moreover, for every \(\delta>0\), the integral is absolutely convergent
and its tails tend to zero uniformly on $
\{z=x+it: x\in\mathbb R,\ |t|\le \pi/2-\delta\}$. 
\end{corollary}
\begin{proof}
We apply Harper's weighted Paley--Wiener theorem for vertical strips
\cite[Theorem~2.1]{Harper}. To fit Harper's convention, we set
\[
G(z)=F(iz),
\qquad
-\frac{\pi}{2}<\operatorname{Re}z<\frac{\pi}{2}.
\]
The corresponding weight is $
v(x)=\cos^\alpha x $. 
The local integrability condition required in Harper's theorem is
immediate. Indeed, for every compact interval $
J\subset(-\pi/2,\pi/2)$,  
the function $v(x)^{-\varepsilon}=\cos^{-\alpha\varepsilon}x$ is bounded on $J$
for every $\varepsilon>0$. Therefore Harper's theorem applies.
Moreover, the associated Laplace weight is
\[
\int_{-\pi/2}^{\pi/2}
e^{-2x\xi}v(x)\,dx
=
\int_{-\pi/2}^{\pi/2}
e^{-2x\xi}\cos^\alpha x\,dx
=
m(\xi).
\]
After translating Harper's Laplace representation back from \(G\) to
\(F\), namely using \(F(u)=G(-iu)\), and taking into account the
normalization constants in the Fourier representation, we obtain
\eqref{eq:fourier}.  The norm identity and uniqueness of the representing
function follow from the corresponding statements in Harper's theorem.

Conversely, the direct part of Harper's theorem implies that every $\varphi\in L^2(\mathbb R,m(\xi)\,d\xi)$ 
gives, through \eqref{eq:fourier}, an element of \(\Hh\) with the same norm.
Hence \eqref{eq:fourier} gives an isometric representation of \(\Hh\) by
\(L^2(\mathbb R,m(\xi)d\xi)\).

Fix \(\delta>0\) and assume \(|t|\le \frac{\pi}{2}-\delta\).  For every
integer \(j\ge0\), the Cauchy--Schwarz inequality gives
\[
 \int_\R|\varphi(\xi)|\,|\xi|^j e^{-t\xi}\,d\xi
 \le
 \|\varphi\|_{L^2(m)}
 \left(
   \int_\R
   \frac{|\xi|^{2j}e^{-2t\xi}}{m(\xi)}\,d\xi
 \right)^{1/2}.
\]
By Proposition~\ref{lem:growth-m}, for sufficiently large \(|\xi|\),
uniformly for \(|t|\le \frac{\pi}{2}-\delta\),
\begin{equation*}\label{l1}
 \frac{|\xi|^{2j}e^{-2t\xi}}{m(\xi)}
 \le
 C_{\delta,j}
 (1+|\xi|)^{2j+\alpha+1}e^{-2\delta|\xi|}.
\end{equation*}
The function on the right belongs to \(L^1(\R)\).  Hence, for every
\(j\ge0\), the integral
\[
 \int_\R
 (i\xi)^j\varphi(\xi)e^{ix\xi-t\xi}\,d\xi
\]
converges absolutely, and its tails tend to zero uniformly for
\(x\in\R\) and \(|t|\le \frac{\pi}{2}-\delta\).  It follows that differentiation
under the integral sign is valid, and
\[
 F^{(j)}(x+it)
 =
 \frac1{\sqrt{2\pi}}
 \int_\R
 (i\xi)^j\varphi(\xi)e^{ix\xi-t\xi}\,d\xi,
 \qquad j\ge0.
\]
Thus these representations converge absolutely and uniformly on every
closed substrip
\[
 \{x+it:x\in\R,\ |t|\le \frac{\pi}{2}-\delta\}.
\]
\end{proof}

\section{Periodic zeros and a minimum-norm problem}\label{sec:lattice}

Fix \(\alpha>1\) and \(\tau>0\).  For \(h>0\) sufficiently small, set
\begin{equation}\label{eq:scaling}
 q=\frac{2\pi}{h},\qquad
 \theta_h=
 \frac{(\alpha+1)\log(2q)
       +\log\!\bigl(\tau/\Gamma(\alpha+1)\bigr)}
      {2q},
 \qquad
 T_h=\frac{\pi}{2}-\theta_h.
\end{equation}
Then
\[
 0<T_h<\frac{\pi}{2},
 \qquad
 T_h\longrightarrow\frac{\pi}{2}
 \quad(h\downarrow0).
\]

\subsection{Periodic zeros}

Define
\[
 \Lambda_h=\{kh+iT_h:k\in\Z\}\cup\{kh-iT_h:k\in\Z\},\qquad
 \Mh=\{F\in\Hh:F(\lambda)=0\text{ for all }\lambda\in\Lambda_h\}.
\]
Each point evaluation on $\Hh$ is bounded, hence $\Mh$ is a closed
subspace.  

\begin{lemma}\label{lem 3.1}
 The space $\Mh\neq \{0\}$ and $K_h(x,0)$ is real for real $x$,
where $K_h$ denotes the reproducing kernel of $\Mh$.
\end{lemma}

\begin{proof}
  Define
\[
 F_h(u)
 =
 \sech^2(u/2)
 \sin\frac{\pi(u-iT_h)}{h}
 \sin\frac{\pi(u+iT_h)}{h}.
\]
The function \(F_h\) is holomorphic on \(S\). If \(u=x+it\in S\), then
\[
 \left|
 \sin\frac{\pi(u\mp iT_h)}{h}
 \right|^2
 =
 \sin^2\frac{\pi x}{h}
 +
 \sinh^2\frac{\pi(t\mp T_h)}{h}
 \le
 \cosh^2\frac{\pi(\pi/2 +T_h)}{h}
\]
and
\[
 \left|\sech^2\frac{u}{2}\right|^2
 =
 \frac{4}{(\cosh x+\cos t)^2}
 \le
 \frac{4}{\cosh^2x}.
\]
Consequently,
\[
 |F_h(x+it)|^2\cos^\alpha t
 \le
 C_h\frac{\cos^\alpha t}{\cosh^2x},
\]
and hence \(F_h\in\Hh\).

Since $\Lambda_h$ is invariant under complex conjugation,
Proposition~\ref{lem:real} shows that $K_h(x,0)$ is real for real $x$.
\end{proof}

For fixed \(\beta\in I_h =(-q/2,q/2]\), let \(\Sigma_{h,\beta}\) be the weighted
\(\ell^2\)-space consisting of all sequences \(c=(c_n)_{n\in\Z}\)
such that
\[
\|c\|_\beta^2
:=
\sum_{n\in\Z}m(\beta+nq)|c_n|^2
<\infty.
\]
Its inner product is
\[
\langle c,b\rangle_\beta
=
\sum_{n\in\Z}
m(\beta+nq)c_n\overline{b_n}.
\]

Let \(\Sigma_h\) be the space of all sequences
\[
a=(a_n)_{n\in\Z},
\]
such that each function \(a_n:I_h\to\C\) is measurable and such that
\[
\int_{I_h}\|a(\beta)\|_\beta^2\,d\beta
=
\int_{I_h}
\sum_{n\in\Z}
m(\beta+nq)|a_n(\beta)|^2\,d\beta
<\infty,
\]
where \(a(\beta)=(a_n(\beta))_{n\in\mathbb Z}\).
 We equip \(\Sigma_h\) with the inner product
\begin{equation}\label{sigmainner}
 \langle a,b\rangle_{\Sigma_h}
=
\int_{I_h}
\langle a(\beta),b(\beta)\rangle_\beta\,d\beta.
\end{equation}
If $a \in \Sigma_h$, then it is clear that $a(\beta) \in \Sigma_{h,\beta}$ for almost all $\beta \in I_h$.

Since the intervals \(I_h+nq\), \(n\in\Z\), form a disjoint partition
of \(\R\), every \(\xi\in\R\) can be written uniquely as
\[
\xi=\beta+nq,
\qquad
\beta\in I_h,\quad n\in\Z.
\]

If \(\varphi\in L^2(\R,m(\xi)\,d\xi)\), then
\begin{align}
\int_\R|\varphi(\xi)|^2m(\xi)\,d\xi
&=
\sum_{n\in\Z}
\int_{I_h+nq}|\varphi(\xi)|^2m(\xi)\,d\xi
\notag\\
&=
\int_{I_h}
\sum_{n\in\Z}
m(\beta+nq)
|\varphi(\beta+nq)|^2\,d\beta.
\label{eq:partition}
\end{align}
Hence
\[
\varphi
\longmapsto
\bigl(\varphi(\beta+nq)\bigr)_{n\in\Z}
\]
is an isometry from
\(L^2(\R,m(\xi)\,d\xi)\) into \(\Sigma_h\).

Conversely, given \(a\in\Sigma_h\), define
\begin{equation}\label{mapsigma}
 \varphi(\beta+nq)=a_n(\beta),
\qquad
\beta\in I_h,\quad n\in\Z.
\end{equation}

Then \(\varphi\in L^2(\R,m(\xi)\,d\xi)\), and
\eqref{eq:partition} gives equality of the norms. Thus the preceding
map is an isometric isomorphism.

Combining this identification with Corollary~\ref{lem:fourier}, we
obtain a unitary map
\[
\mathcal F_h:\Hh\longrightarrow\Sigma_h.
\]
More precisely, if \(F\in\Hh\) is written as in
\eqref{eq:fourier} with
\(\varphi\in L^2(\R,m(\xi)\,d\xi)\), then
\begin{equation}\label{eq:Fh-map}
(\mathcal F_hF)_n(\beta)
=
\varphi(\beta+nq),
\qquad
\beta\in I_h,\quad n\in\Z.
\end{equation}

\begin{proposition}\label{lem:sampling}
	Let \(F\in\Hh\), and write $
	a=\mathcal F_hF\in\Sigma_h$. 
	Then \(F\in\Mh\) if and only if, for almost every
	\(\beta\in I_h\),
	\begin{equation}\label{eq:constraints}
		\sum_{n\in\Z}a_n(\beta)e^{T_h(\beta+nq)}=0,
		\qquad
		\sum_{n\in\Z}a_n(\beta)e^{-T_h(\beta+nq)}=0.
	\end{equation}
	Both series are absolutely convergent for almost every
	\(\beta\in I_h\).
	
	Consequently,
	\[
	\mathcal F_h(\Mh)=V_h,
	\]
	where
	\[
	V_h
	:=
	\left\{
	a\in\Sigma_h:
	a \text{ satisfies \eqref{eq:constraints} for almost every }
	\beta\in I_h
	\right\}.
	\]
	In particular, the restriction $
	\mathcal F_h:\Mh\longrightarrow V_h$ 
	is unitary.
\end{proposition}

\begin{proof}
	Fix \(h>0\) and \(t\) with \(|t|<\pi/2\).
		Let $Q_t$ denote a sequence of functions on $I_h$ defined by
		\begin{equation}\label{qt}
			Q_t(\beta)=
			\left(
			\frac{e^{-t(\beta+nq)}}{m(\beta+nq)}
			\right)_{n\in\Z},
			\qquad \beta\in I_h.
		\end{equation}
		We first verify that $\|Q_t(\beta)\|_{\beta}$ is bounded on $I_h$.
		By Proposition~\ref{lem:growth-m}, we have
		\begin{equation}\label{eq:Qt-majorant}
			\frac{e^{-2t(\beta+nq)}}{m(\beta+nq)}
			\le
			C_{t,h} (1+|n|q)^{\alpha+1}
			e^{-(\pi-2|t|)|n|q}.
		\end{equation}
		Since $
		|t|<\frac{\pi}{2}$,
		we have $
		\pi-2|t|>0$, and so
		\begin{equation*}
			\sum_{|n|>N}
			\frac{e^{-2t(\beta+nq)}}{m(\beta+nq)}
			\leq
			C_{t,h} \sum_{|n|>N}
			(1+|n|q)^{\alpha+1}
			e^{-(\pi-2|t|)|n|q}
			\to 0
			\quad \text{as}  \, N\to \infty.
		\end{equation*}
		Hence, we have
		\begin{equation}\label{Qt}
			\|Q_t(\beta)\|_{\beta}^2
			=
			\sum_{n\in \Z}
			\frac{e^{-2t(\beta+nq)}}{m(\beta+nq)}
			=
			\lim_{N\to \infty}
			\sum_{|n|\leq N}
			\frac{e^{-2t(\beta+nq)}}{m(\beta+nq)},
		\end{equation}
		where the limit above is uniformly for $\beta \in I_h$.
		Since the finite sums in \eqref{Qt} are bounded and
		continuous functions on $I_h$, the same is true of
		$\|Q_t(\beta)\|_{\beta}$.

	By \eqref{eq:partition},
	\[
	\|a(\beta)\|_\beta<\infty
	\]
	for almost every \(\beta\in I_h\).
		We therefore obtain
		\[
		\sum_{n\in\Z}
		|a_n(\beta)|e^{-t(\beta+nq)}
		=
		\langle |a|, Q_t(\beta)\rangle_{\beta}
		\le
		\|a(\beta)\|_\beta \|Q_t(\beta)\|_\beta,
		\]
		where $|a|$ denotes the sequence
		$(|a_n|)_{n\in\Z}$.
		Thus
		\[
		B_t(\beta):=
		\sum_{n\in\Z}a_n(\beta)e^{-t(\beta+nq)}
		\]
		is absolutely convergent for almost every \(\beta\), and
		\[
		|B_t(\beta)|^2
		\le
		\|Q_t(\beta)\|_\beta^2\|a(\beta)\|_\beta^2.
		\]
		Using the boundedness of $\|Q_t(\beta)\|_\beta$ on
		\(I_h\) and \eqref{sigmainner}, we obtain
		\[
		\int_{I_h}|B_t(\beta)|^2\,d\beta
		\le
		\left(
		\sup_{\beta\in I_h}\|Q_t(\beta)\|_\beta^2
		\right)
		\int_{I_h}\|a(\beta)\|_\beta^2\,d\beta
		=
		\left(
		\sup_{\beta\in I_h}\|Q_t(\beta)\|_\beta^2
		\right)
		\|a\|_{\Sigma_h}^2.
		\]
		Thus $B_t\in L^2(I_h)$.

	By Corollary~\ref{lem:fourier} and \(hq=2\pi\), we have
	\[
	\begin{aligned}
		F(kh+it)
		&=
		\frac1{\sqrt{2\pi}}
		\sum_{n\in\Z}\int_{I_h}
		a_n(\beta)e^{ikh(\beta+nq)}
		e^{-t(\beta+nq)}\,d\beta\\
		&=
		\frac1{\sqrt{2\pi}}
		\int_{I_h}e^{2\pi ik\beta/q}B_t(\beta)\,d\beta.
	\end{aligned}
	\]
		By the uniqueness of Fourier coefficients in $L^2(I_h)$,
		we conclude that
	\[
	F(kh+it)=0
	\quad\text{for every }k\in\Z
	\]
	if and only if
	\[
	B_t(\beta)=0
	\quad\text{for almost every }\beta\in I_h.
	\]
	Finally, apply this equivalence first with $t=T_h$ and then with
	$t=-T_h$. These two choices give exactly the two identities in
	\eqref{eq:constraints}. The absolute convergence asserted in the
	statement has already been proved above.
\end{proof}

The preceding proposition identifies the periodic vanishing conditions
with two linear constraints.  For fixed $\beta\in I_h$, let $V_{h,\beta}$ denote the closed subspace of the weighted $\ell^2-$space  \(\Sigma_{h,\beta}\) consisting of those sequences that satisfy \eqref{eq:constraints}.  We next relate this description to the
reproducing kernel $K_h$ of \(\Mh\).

For \(z\in S\) and \(\beta\in I_h\), define
\[
(r_{z,\beta})_n
=
\frac{e^{-i\overline z(\beta+nq)}}{m(\beta+nq)},
\qquad n\in\Z,
\]
and set
\[
r_z(\beta)
=
\frac1{\sqrt{2\pi}}\,r_{z,\beta}.
\]
By the estimates for \(Q_t\) in the proof of
Proposition~\ref{lem:sampling}, \(r_z\in\Sigma_h\). 
Let \(F\in\Mh\), \(z\in S\), and set \(a=\mathcal F_hF\in V_h\).
By Corollary~\ref{lem:fourier},
\[
 \langle F, K_{h}(\cdot, z)\rangle=F(z)=\langle a, r_z\rangle_{\Sigma_h}=\langle a, P_{V_h}r_z\rangle_{\Sigma_h},
 \]
where \(P_{V_h}\) denotes the orthogonal projection of
\(\Sigma_h\) onto \(V_h\).  Since $\mathcal F_h:\Mh\longrightarrow V_h$ is unitary, we have 
\(K_{h}(\cdot, z)=\mathcal F_h^{-1}P_{V_h}r_z\). Therefore, for $z, w\in S$, we have
\begin{equation}\label{kernel}
 K_h(w, z)=\langle K_{h}(\cdot, z), K_{h}(\cdot, w)\rangle=\langle P_{V_h}r_z, P_{V_h}r_w\rangle.
\end{equation}
\begin{lemma}\label{lemma:proj}
  For $a\in \Sigma_h$, $(P_{V_h}a)(\beta)=P_{V_{h, \beta}}(a(\beta))$ for almost all $\beta\in I_h$.
\end{lemma}
\begin{proof}
We first prove that $
\beta\longmapsto P_{V_{h,\beta}}(a(\beta))$ 
defines an element of \(\Sigma_h\).

For each fixed \(\beta\in I_h\), define the operator $C_\beta: \Sigma_{h,\beta} \to \mathbb{C}^2$ by
\[
C_\beta c
=
\begin{pmatrix}
\displaystyle\sum_{n\in\Z}
c_ne^{T_h(\beta+nq)}
\\[2mm]
\displaystyle\sum_{n\in\Z}
c_ne^{-T_h(\beta+nq)}
\end{pmatrix}.
\]
Then
\[
V_{h,\beta}=\ker C_\beta.
\]
As in the proof of Proposition~\ref{lem:sampling}, \(C_\beta\) is
bounded. It is also onto. Indeed, the sequences supported at \(n=1\)
and \(n=-1\) are mapped respectively to
\[
\begin{pmatrix}
e^{T_h(\beta+q)}\\
e^{-T_h(\beta+q)}
\end{pmatrix},
\qquad
\begin{pmatrix}
e^{T_h(\beta-q)}\\
e^{-T_h(\beta-q)}
\end{pmatrix}.
\]
The determinant of these two vectors is
\[
2\sinh(2T_hq)\neq0.
\]
Hence \(C_\beta\) is surjective. Hence
\[
P_{V_{h,\beta}}
=
I-C_\beta^*(C_\beta C_\beta^*)^{-1}C_\beta.
\]

Since the norm of \(\Sigma_{h,\beta}\) depends on \(\beta\), define
\[
U_\beta c
=
\left(
\sqrt{m(\beta+nq)}\,c_n
\right)_{n\in\Z}.
\]
Then \(U_\beta\) is unitary from \(\Sigma_{h,\beta}\) onto
\(\ell^2(\Z)\). Set
\[
\widetilde C_\beta=C_\beta U_\beta^{-1}.
\]
Thus
\[
\widetilde C_\beta x
=
\begin{pmatrix}
\langle x,u_+(\beta)\rangle_{\ell^2}\\
\langle x,u_-(\beta)\rangle_{\ell^2}
\end{pmatrix},
\]
where
\[
u_\pm(\beta)
=
\left(
\frac{e^{\pm T_h(\beta+nq)}}
{\sqrt{m(\beta+nq)}}
\right)_{n\in\Z}.
\]

The \(Q_t\)-estimate in the proof of
Proposition~\ref{lem:sampling} implies
\[
\sup_{\beta\in I_h}
\sum_{|n|>N}
\frac{e^{\pm2T_h(\beta+nq)}}{m(\beta+nq)}
\longrightarrow0
\qquad(N\to\infty).
\]
Since each coordinate depends continuously on \(\beta\), it follows that
\(\beta\mapsto u_\pm(\beta)\) is continuous as an
\(\ell^2(\Z)\)-valued function. Hence
\(\beta\mapsto\widetilde C_\beta\) is continuous in operator norm.
Therefore
\[
\widetilde P_\beta
:=
U_\beta P_{V_{h,\beta}}U_\beta^{-1}
=
I-\widetilde C_\beta^*
(\widetilde C_\beta\widetilde C_\beta^*)^{-1}
\widetilde C_\beta
\]
depends continuously on \(\beta\) in operator norm.

Since \(\beta\mapsto U_\beta a(\beta)\) is measurable,
\[
\beta\longmapsto
\widetilde P_\beta U_\beta a(\beta)
\]
is measurable.   Moreover,
\[
\int_{I_h}
\|P_{V_{h,\beta}}a(\beta)\|_\beta^2\,d\beta
\le
\int_{I_h}\|a(\beta)\|_\beta^2\,d\beta
<\infty.
\]
Hence $
\beta\longmapsto P_{V_{h,\beta}}a(\beta)$ 
defines an element of \(\Sigma_h\), and it belongs to \(V_h\).

Now let \(b\in V_h\). Then
\(b(\beta)\in V_{h,\beta}\) for almost every \(\beta\), and therefore
\[
\begin{aligned}
&\int_{I_h}
\left\langle
a(\beta)-P_{V_{h,\beta}}a(\beta),
b(\beta)
\right\rangle_\beta\,d\beta  =0.
\end{aligned}
\]
Thus
\[
a-\bigl(P_{V_{h,\beta}}a(\beta)\bigr)_{\beta\in I_h}
\perp V_h,
\]
and hence
\[
(P_{V_h}a)(\beta)
=
P_{V_{h,\beta}}a(\beta)
\]
for almost every \(\beta\in I_h\).
\end{proof}
 
For \(z,w\in S\), \eqref{kernel} and
Lemma~\ref{lemma:proj} give
\begin{equation}\label{eq:global-kernel}
K_h(z,w)
=
\frac1{2\pi}
\int_{I_h}
k_h(\beta;z, w)\,d\beta,
\end{equation}
where
\begin{equation}\label{eq:fiber-kernel}
k_h(\beta;z, w)
:=
\left\langle
P_{V_{h,\beta}}r_{w,\beta},
P_{V_{h,\beta}}r_{z,\beta}
\right\rangle_\beta.
\end{equation}
Thus the study of the global kernel \(K_h(x,0)\) is reduced to that of
\(k_h(\beta;x,0)\) on \(V_{h,\beta}\).

For \(a\in\Sigma_{h,\beta}\), we have
\[
\langle a,r_{x,\beta}\rangle_\beta
=
a_0e^{i\beta x}
+
\sum_{n\ne0}a_ne^{i(\beta+nq)x}.
\]
Hence the Cauchy--Schwarz inequality gives
\begin{equation}\label{eq:evaluation-zero-coordinate}
\left|
\langle a,r_{x,\beta}\rangle_\beta
-
a_0e^{i\beta x}
\right|^2
\le
\|a\|_\beta^2
\sum_{n\ne0}\frac1{m(\beta+nq)}.
\end{equation}
As will be shown in \eqref{eq:Rh-estimate}, the sum on the right tends
to zero as \(h\downarrow0\), uniformly for \(\beta\) in every fixed
compact interval.  Thus, as \(h\downarrow0\), the restriction of the
functional represented by \(r_{x,\beta}\) to \(V_{h,\beta}\) is governed
to leading order by the zeroth-coordinate functional
\[
a\longmapsto a_0 e^{i\beta x}.
\]

We are therefore led to determine the Riesz representing vector of
\(a\mapsto a_0\) on \(V_{h,\beta}\).  Equivalently, we consider the
minimum-norm problem
\[
\min\left\{
\|a\|_\beta^2:
a\in V_{h,\beta},\ a_0=1
\right\}.
\]
The next proposition solves this minimum-norm problem and determines
its asymptotic behavior.
 
\subsection{The minimum-norm problem}

We first record the Hilbert space fact needed to carry out the
preceding construction.  If \(C:H\to\C^d\) is bounded and surjective,
then \(CC^*\) is invertible and the unique minimum-norm solution of
\(Cx=y\) is
\begin{equation}\label{eq:minimum-norm-formula}
 x_0=C^*(CC^*)^{-1}y,
 \qquad
 \|x_0\|^2=y^*(CC^*)^{-1}y.
\end{equation}
Indeed,
\begin{equation}\label{eq:kernel-range-decomposition}
 H=\ker C\oplus\Ran C^*,
\end{equation}
and the assertion follows from this orthogonal decomposition.

We now apply this fact to \(V_{h,\beta}\).
Suppose that the zeroth coordinate \(a_0\) is prescribed for $a\in V_{h, \beta}$, and write
\(b_n=a_n\) for \(n\ne0\), with \(b_0=0\). Then the two constraints
in \eqref{eq:constraints} become
\[
 Cb
 =
 -a_0
 \begin{pmatrix}
  e^{T_h\beta}\\
  e^{-T_h\beta}
 \end{pmatrix},
\]
where
\[
 C:
 \{b\in \Sigma_{h, \beta}: \ b_0=0\}
 \longrightarrow\C^2
\]
is given by
\begin{equation}\label{eq:C-definition}
 Cb
 =
 \begin{pmatrix}
  \displaystyle\sum_{n\ne0}b_ne^{T_h(\beta+nq)}\\[2mm]
  \displaystyle\sum_{n\ne0}b_ne^{-T_h(\beta+nq)}
 \end{pmatrix}.
\end{equation}

The corresponding matrix \(CC^*\) is
\begin{equation}\label{eq:Gram}
 S_h(\beta)
 =
 \begin{pmatrix}
 \displaystyle\sum_{n\ne0}
 \frac{e^{2T_h(\beta+nq)}}{m(\beta+nq)}
 &
 \displaystyle\sum_{n\ne0}\frac1{m(\beta+nq)}
 \\[3mm]
 \displaystyle\sum_{n\ne0}\frac1{m(\beta+nq)}
 &
 \displaystyle\sum_{n\ne0}
 \frac{e^{-2T_h(\beta+nq)}}{m(\beta+nq)}
 \end{pmatrix}.
\end{equation}
For later use, set
\begin{equation}\label{eq:D}
 D_\tau(z):=m(z)+2\tau\cosh(\pi z).
\end{equation}

\begin{proposition}\label{lem:min}\label{lem:Gram-limit}
The matrix \(S_h(\beta)\) is positive definite.  In \(V_{h,\beta}\)
there is a unique minimum-norm vector \(g_{h,\beta}\) satisfying $
 (g_{h,\beta})_0=1$, and 
\begin{equation}\label{eq:dh}
\|g_{h,\beta}\|_\beta^2
=
m(\beta)
+
\begin{pmatrix}
 e^{T_h\beta} & e^{-T_h\beta}
\end{pmatrix}
S_h(\beta)^{-1}
\begin{pmatrix}
 e^{T_h\beta}\\
 e^{-T_h\beta}
\end{pmatrix}.
\end{equation}
Moreover, if $
 W_{h,\beta}
 =
 \{a\in V_{h,\beta}:a_0=0\}$, 
then
\begin{equation}\label{eq:V-decomposition}
 V_{h,\beta}
 =
 \spanop\{g_{h,\beta}\}
 \oplus W_{h,\beta}
\end{equation}
orthogonally, and
\begin{equation}\label{eq:zeroth-coordinate-representer}
 a_0
 =
 \left\langle
 a,\frac{g_{h,\beta}}{\|g_{h,\beta}\|_\beta^2}
 \right\rangle_{\beta},
 \qquad a\in V_{h,\beta}.
\end{equation}
Finally, as \(h\downarrow0\),
\begin{equation}\label{dnbeta}
S_h(\beta)\longrightarrow\tau^{-1}I_2,
\qquad
\|g_{h,\beta}\|_\beta^2\longrightarrow D_\tau(\beta),
\end{equation}
uniformly on every fixed compact interval of \(\beta\).
\end{proposition}

\begin{proof} 
Fix \(h>0\) and \(\beta\in I_h\). To simplify notation, we omit the
dependence on \(h\) and \(\beta\) in this part of the proof. 
Let \(C\) be the operator defined in \eqref{eq:C-definition}.  For $b\in \Sigma_{h, \beta}$ with \(b_0=0\), 
\[
 \left|\sum_{n\ne0}b_ne^{T_h(\beta+nq)}\right|= |\langle b, Q_{-T_h}(\beta)\rangle_{\beta}|
 \le \|Q_{-T_h}(\beta)\|_{\beta} \|b\|_{\beta},
\]
and similarly,
\[
 \left|\sum_{n\ne0}b_ne^{-T_h(\beta+nq)}\right|
 \le \|Q_{T_h}(\beta)\|_{\beta} \|b\|_{\beta},
\]
where $Q_t$ defined in \eqref{qt}. Hence \(C\) is bounded.
As in the argument in Lemma \ref{lemma:proj}, one knows that \(C\) is onto.

A direct computation with the weighted inner
product gives
\[
 CC^*=S_h(\beta),
\]
where \(S_h(\beta)\) is given by \eqref{eq:Gram}.  Thus
\(S_h(\beta)\) is positive definite.

Suppose now that the zeroth coordinate \(a_0\) is prescribed.
The remaining coordinates satisfy
\[
 Cb
 =
 -a_0
 \begin{pmatrix}
  e^{T_h\beta}\\
  e^{-T_h\beta}
 \end{pmatrix}.
\]
By \eqref{eq:minimum-norm-formula}, the minimum-norm solution is
\[
 b
 =
 -a_0 C^*S_h(\beta)^{-1}
 \begin{pmatrix}
  e^{T_h\beta}\\
  e^{-T_h\beta}
 \end{pmatrix},
\]
and
\[
 \|b\|_\beta^2
 =
 |a_0|^2
 \begin{pmatrix}
  e^{T_h\beta}&e^{-T_h\beta}
 \end{pmatrix}
 S_h(\beta)^{-1}
 \begin{pmatrix}
  e^{T_h\beta}\\
  e^{-T_h\beta}
 \end{pmatrix}.
\]
Taking \(a_0=1\) gives \eqref{eq:dh}.  By \eqref{eq:kernel-range-decomposition},
the closed subspace $
 \{b\in \Sigma_{h, \beta}: \ b_0=0\}$ 
decomposes as $
 \ker C\oplus\Ran C^*$ 
and hence the orthogonal decomposition
\eqref{eq:V-decomposition} holds.

Since \((g_{h,\beta})_0=1\), every
\(a\in V_{h,\beta}\) can be written uniquely as
\[
 a=a_0g_{h,\beta}+w,
 \qquad w\in W_{h,\beta}.
\]
The orthogonality above therefore gives
\[
 \langle a,g_{h,\beta}\rangle_{\beta}
 =
 a_0\|g_{h,\beta}\|_{\beta}^2,
\]
and hence
\[
a_0
=
\left\langle
a,\frac{g_{h,\beta}}{\|g_{h,\beta}\|_{\beta}^2}
\right\rangle_\beta,
\]
which proves \eqref{eq:zeroth-coordinate-representer}.
 
From the definition of \(\theta_h\) in \eqref{eq:scaling},
\begin{equation}\label{eq:theta-identity}
 e^{-2q\theta_h}
 =
 \frac{\Gamma(\alpha+1)}
      {\tau(2q)^{\alpha+1}}.
\end{equation}

Fix \(R>0\) and for sufficiently small $h>0$, we have $q>2R$.  Then, for
\(|\beta|\le R\), the sign of $
 (\beta+nq)=\beta+nq$ 
agrees with the sign of \(n\) whenever \(n\ne0\).

Consider first the upper-left entry of \(S_h(\beta)\):
\[
 \sum_{n\ne0}
 \frac{e^{2T_h(\beta+nq)}}{m(\beta+nq)}.
\]
For \(n\ge1\), Proposition~\ref{lem:growth-m} and
\(T_h=\frac{\pi}{2}-\theta_h\) give, uniformly for \(|\beta|\le R\),
\begin{equation}\label{eq:positive-mode-asymptotic}
 \frac{e^{2T_h(\beta+nq)}}{m(\beta+nq)}
 =
 \frac{2^{\alpha+1}}{\Gamma(\alpha+1)}
 (\beta+nq)^{\alpha+1}e^{-2\theta_h(\beta+nq)}
 \bigl(1+O(q^{-2})\bigr).
\end{equation}

For \(n=1\),  
\eqref{eq:theta-identity} gives
\[
\begin{aligned}
 \frac{e^{2T_h(\beta+q)}}{m(\beta+q)}
 &=
 \frac{2^{\alpha+1}}{\Gamma(\alpha+1)}
 (q+\beta)^{\alpha+1}
 e^{-2\theta_h(q+\beta)}
 \bigl(1+O(q^{-2})\bigr)\\
 &=
 \tau^{-1}
 \left(1+\frac{\beta}{q}\right)^{\alpha+1}
 e^{-2\theta_h\beta}
 \bigl(1+O(q^{-2})\bigr).
\end{aligned}
\]
As \(q\to\infty\) and \(\theta_h\to0\), this converges to
\(\tau^{-1}\) uniformly for \(|\beta|\le R\).

We next estimate the terms with \(n\ge2\).  From
\eqref{eq:positive-mode-asymptotic},  
\[
 \sum_{n\ge2}
 \frac{e^{2T_h(\beta+nq)}}{m(\beta+nq)}
 \le
 C_R q^{\alpha+1}
 \sum_{n\ge2}
 n^{\alpha+1} e^{-2nq\theta_h}.
\]
By \eqref{eq:theta-identity},
\[
 e^{-2q\theta_h}
 =
 O(q^{-\alpha-1})\longrightarrow0.
\]
Hence, for all sufficiently small \(h\),
\(e^{-2q\theta_h}\le1/2\).  Therefore
\[
 \begin{aligned}
 \sum_{n\ge2}n^{\alpha+1} e^{-2nq\theta_h}
 &=
 e^{-4q\theta_h}
 \sum_{n\ge2}
 n^{\alpha+1} e^{-2(n-2)q\theta_h}\\
 &\le
 e^{-4q\theta_h}
 \sum_{n\ge2}n^{\alpha+1}2^{-(n-2)}
 =
 C_\alpha e^{-4q\theta_h}.
 \end{aligned}
\]
Using \eqref{eq:theta-identity} again,
\[
 q^{\alpha+1} e^{-4q\theta_h}
 =
 q^{\alpha+1}
 \left(
\frac{\Gamma(\alpha+1)}
     {\tau(2q)^{\alpha+1}}
 \right)^2
 =
 O(q^{-\alpha-1})
 \longrightarrow0.
\]
Thus
\[
 \sum_{n\ge2}
 \frac{e^{2T_h(\beta+nq)}}{m(\beta+nq)}
 \longrightarrow0
\]
uniformly for \(|\beta|\le R\).

For \(n\le-1\),  \(\beta+nq<0\).  By \eqref{eq:laplace},
\[
 \frac{e^{2T_h(\beta+nq)}}{m(\beta+nq)}
 \le C_R (1+|\beta+nq|)^{\alpha+1}
 e^{-\pi|\beta+nq|}
.
\]
Since $
 |\beta+nq|\ge (|n|-1/2)q$, 
the sum over \(n\le-1\) tends to zero uniformly for
\(|\beta|\le R\).  Hence the upper-left entry of
\(S_h(\beta)\) converges uniformly to \(\tau^{-1}\) on
\([-R,R]\).

 By the evenness of $m$ and the change of index $n\mapsto -n$,
\[
 (S_h(\beta))_{22}=(S_h(-\beta))_{11}.
\]
Hence the lower-right entry also converges to $\tau^{-1}$,
uniformly for $\beta$ in compact intervals.

The two off-diagonal entries are equal to
\begin{equation}\label{eq:Rh}
 R_h(\beta):=\sum_{n\ne0}\frac1{m(\beta+nq)}.
\end{equation}
By \eqref{eq:Qt-majorant} with \(t=0\), for \(|\beta|\le R\) and
\(n\ne0\),
\[
 \frac1{m(\beta+nq)}
 \le
 C_R (1+|n|q)^{\alpha+1} e^{-\pi|n|q}.
\]
Hence, for all sufficiently small \(h\),
\[
 \begin{aligned}
 R_h(\beta)
 &\le
 C_R q^{\alpha+1}
 \sum_{n\ne0}(1+|n|)^{\alpha+1} e^{-\pi|n|q}\\
 &\le
 C_R q^{\alpha+1} e^{-\pi q}
 \sum_{n\ne0}(1+|n|)^{\alpha+1}
 e^{-\pi(|n|-1)q}.
 \end{aligned}
\]
Taking \(q\) large enough, we may assume \(e^{-\pi q}\le1/2\).  Therefore
\[
 \sum_{n\ne0}(1+|n|)^{\alpha+1}
 e^{-\pi(|n|-1)q}
 \le
 2\sum_{n\ge1}(1+n)^{\alpha+1}2^{-(n-1)}
 <\infty.
\]
Consequently,
\begin{equation}\label{eq:Rh-estimate}
 \sup_{|\beta|\le R}R_h(\beta)
 =O(q^{\alpha+1}e^{-\pi q})\longrightarrow0.
\end{equation}
Thus
\[
 S_h(\beta)\longrightarrow\tau^{-1}I_2
\]
uniformly on $[-R,R]$.

Using the continuity of operator inversion, we therefore obtain
$
 S_h(\beta)^{-1}\longrightarrow\tau I_2$ 
uniformly on \([-R,R]\).  Finally,
\[
 \begin{pmatrix}
 e^{T_h\beta}\\
 e^{-T_h\beta}
 \end{pmatrix}
 \longrightarrow
 \begin{pmatrix}
 e^{\pi\beta/2}\\
 e^{-\pi\beta/2}
 \end{pmatrix}
\]
uniformly for \(|\beta|\le R\).  Substituting this and
\(S_h(\beta)^{-1}\to\tau I_2\) into \eqref{eq:dh}, we obtain
\[
\|g_{h,\beta}\|_{\beta}^2\longrightarrow D_\tau(\beta).
\]
This completes the proof of Proposition~\ref{lem:Gram-limit}.
\end{proof}

\subsection{Limit of the reproducing kernels}
\label{sec:kernels}

We now use the minimum-norm vector to compute the leading term of
\(k_h(\beta;x,0)\).  Define \(r_{x,\beta}'\in \Sigma_{h,\beta}\) by
\[
(r_{x,\beta}')_0=(r_{x,\beta})_0,\qquad
(r_{x,\beta}')_n=0\quad(n\neq 0),
\]
and set
\[
r_{x,\beta}''=r_{x,\beta}-r_{x,\beta}'.
\] 
For $a\in V_{h,\beta}$, by \eqref{eq:zeroth-coordinate-representer}, we have
\[
\langle a, P_{V_{h, \beta}}r_{x,\beta}'\rangle_{\beta}= \langle a, r_{x,\beta}'\rangle_{\beta} =e^{i\beta x} a_0
=e^{i\beta x}\left\langle
a,\frac{g_{h,\beta}}{\|g_{h,\beta}\|_{\beta}^2}
\right\rangle_\beta,\]
and hence
\[
P_{V_{h, \beta}}r_{x,\beta}'= e^{-i\beta x} \frac{g_{h,\beta}}{\|g_{h,\beta}\|_{\beta}^2}.
\]
It follows from \eqref{eq:fiber-kernel} that
\begin{align}\label{kerneldecomp}
 k_h(\beta; x, 0)
&= \left\langle P_{V_{h,\beta}}r_{0,\beta}, P_{V_{h,\beta}}r_{x,\beta} \right\rangle_\beta \nonumber \\
&= \langle P_{V_{h,\beta}}r_{0,\beta}', P_{V_{h,\beta}}r_{x,\beta}' \rangle_\beta
 + \langle P_{V_{h,\beta}}r_{0,\beta}', r_{x,\beta}'' \rangle_\beta
 + \langle P_{V_{h,\beta}}r_{0,\beta}'',r_{x,\beta} \rangle_\beta \nonumber \\
&=  \frac{e^{i\beta x}}{\|g_{h,\beta}\|_{\beta}^2}
 +  \langle P_{V_{h,\beta}}r_{0,\beta}', r_{x,\beta}'' \rangle_\beta
 + \langle P_{V_{h,\beta}}r_{0,\beta}'',r_{x,\beta} \rangle_\beta  .
\end{align}
Recall the estimation of $R_h$ \eqref{eq:Rh-estimate}, we have
\begin{equation}\label{rtwo}
\|r_{x,\beta}'' \|_\beta=R_h(\beta)^{1/2}\to 0 \qquad(h\downarrow0),
\end{equation}
uniformly on a compact interval of $\mathbb{R}$, 
and
\begin{equation}\label{rone}
 \|r_{x,\beta} \|_\beta^2=\frac{1}{m(\beta)}+R_h(\beta)
\end{equation}
is bounded uniformly on a compact interval of $\mathbb{R}$. Consequently, the absolute value of the sum of the last two terms in \eqref{kerneldecomp} therefore tends to 
$0$ uniformly on compact intervals of $\mathbb{R}$.
Combining this with \(\|g_{h,\beta}\|_{\beta}^2\to D_\tau(\beta)\), uniformly on
compact intervals, \eqref{kerneldecomp} implies, for each $R>0$,
\begin{equation}\label{eq:fiber-kernel-limit}
 \sup_{|\beta|\le R}
 \left|
 k_h(\beta;x,0)
 -
 \frac{e^{i\beta x}}{D_\tau(\beta)}
 \right|
 \longrightarrow0
 \qquad(h\downarrow0).
\end{equation}

Recalling the representation of the kernel $K_h$ in  \eqref{eq:global-kernel},
the preceding analysis suggests defining the limiting kernel
\begin{equation}\label{eq:Ktau}
 K_\tau(x)
 =
 \frac1{2\pi}
 \int_\R
 \frac{e^{ix\xi}}{D_\tau(\xi)}\,d\xi,
 \qquad x\in\R.
\end{equation}
The integral is absolutely convergent since \(
 D_\tau(\xi)\ge m(\xi)\) for each \(\xi\in\R\) and \(1/m\in L^1(\R)\).
\begin{theorem}\label{thm:kernel-limit}
Fix \(\tau>0\). Then, 
\[
K_h(x,0) \longrightarrow K_\tau(x)
 \qquad(h\downarrow0)
\]
uniformly for $x\in\mathbb{R}$.
\end{theorem}
\begin{proof}
By \eqref{kerneldecomp}, we have
\begin{align}\label{limit}
2\pi |K_h(x,0)- K_\tau(x)| &= \left|\int_{I_h}k_h(\beta; x, 0)\,d\beta-\int_{\R} \frac{e^{ix\xi}}{D_\tau(\xi)}\,d\xi\right|\notag \\
  &\leq \int_{\R}\left| \frac{\chi_{I_h}(\xi)}{\|g_{h,\xi}\|_{\xi}^2}- \frac{1}{D_\tau(\xi)}\right|\,d\xi \notag\\
  & +\int_{I_h}|\langle P_{V_{h,\beta}}r_{0,\beta}', r_{x,\beta}'' \rangle_\beta|\,d\beta 
  +\int_{I_h}| \langle P_{V_{h,\beta}}r_{0,\beta}'',r_{x,\beta} \rangle_\beta|\,d\beta
 \end{align}
 
We first estimate the last two terms on the right-hand side of inequality \eqref{limit}. Using \eqref{rone} and \eqref{rtwo}, by the Cauchy-Schwarz inequality. we have
\begin{align*}\label{tail1}
\int_{I_h}| \langle P_{V_{h,\beta}}r_{0,\beta}'',r_{x,\beta} \rangle_\beta|\,d\beta  & 
 \leq  \int_{I_h}|\|r_{0,\beta}''\|_{\beta} \|r_{x,\beta}\|_{\beta}\,d\beta \notag\\
  & \leq \left(\int_{I_h}\sum_{n\ne0}\frac1{m(\beta+nq)}\,d\beta\right)^{1/2} \left(\int_{I_h}\sum_{n\in \Z}\frac1{m(\beta+nq)}\,d\beta\right)^{1/2} \notag\\
  & = \left(\int_{|\xi|>q/2}\frac{1}{m(\xi)}\,d\xi \right)^{1/2} \left(\int_{\R}\frac{1}{m(\xi)}\,d\xi\right)^{1/2}. 
\end{align*}
The last equation above comes from \eqref{eq:partition}. As $h\downarrow 0$ (so that $q\to\infty$),  the right-hand side of the above inequality tends to $0$, since $1/m\in L^1(\R)$. Thus, the third term of the right-hand side of \eqref{limit} converges uniformly to $0$  in $x\in \R$. The same argument applies to the second term, which therefore converges uniformly to $0$ as $h\downarrow 0$, uniformly in $x\in \R$.

To estimate the first terms of the right-hand side of inequality \eqref{limit}, it follows from \eqref{eq:D} and \eqref{eq:dh} that 
\[
D_\tau(\xi)\geq m(\xi) \quad \xi\in\R
\]
and
\[
 \|g_{h,\xi}\|_{\xi}^2\geq m(\xi) \quad \xi\in I_h,
 \]
 and hence, the integrand is bounded above by $2/m\in L^1(\R)$. For sufficiently small $h$, we have $\xi\in I_h$; then, by \eqref{dnbeta}, the integrand tends to $0$ for each  $\xi\in \R$. Thus, by the dominated convergence theorem, the integral tends to
$0$  as $h\downarrow 0$. 

Since the convergences discussed above are all independent of $x\in \R$, inequality \eqref{limit} implies that $K_h(x,0) \to K_\tau(x)$ as $h\downarrow 0$, uniformly in $x\in \R$.
\end{proof}

\begin{remark}
Our periodic zero set is a discrete analogue of the
two-ray configuration proposed by Hedenmalm and
Perdomo \cite[Section~4]{HP}. Under the conformal map
$u\mapsto ie^u$, the points $kh\pm iT_h$ lie on
two rays symmetric about the imaginary axis, with
consecutive moduli in the ratio $e^h$.
Unlike the continuous configuration considered in
\cite{HP}, our periodic construction allows the
vanishing conditions to be reduced to the two
linear constraints \eqref{eq:constraints},
leading to an explicit integral representation
of the  reproducing kernel.
\end{remark}

\section{The limiting kernel and the sign change}\label{sec:negative}

We shall first show that, for a suitable choice of \(\tau>0\),
the function \(K_\tau\) takes a negative value. 

\subsection{Zeros of \texorpdfstring{$D_\tau$}{D tau}}

For the analysis of the zeros of \(D_\tau\), we shall use an explicit
formula for \(m\).

\begin{lemma}\label{lem:m-gamma}
For every \(z\in\C\),
\begin{equation}\label{eq:gamma}
 m(z)
 =
 \frac{\pi\Gamma(\alpha+1)}
 {2^\alpha
  \Gamma\!\left(\frac{\alpha+2}{2}+iz\right)
  \Gamma\!\left(\frac{\alpha+2}{2}-iz\right)}.
\end{equation}
\end{lemma}

\begin{proof}
Under the Fourier representation \eqref{eq:fourier}, evaluation on $\Hh$ at
the origin corresponds to the functional on $L^2(\mathbb{R}, m(\xi)d\xi)$ as
\[
 \varphi
 \longmapsto
 \frac1{\sqrt{2\pi}}
 \int_\R\varphi(\xi)\,d\xi=\left\langle \varphi, \frac{1}{\sqrt{2\pi}m}\right\rangle_{L^2(\mathbb{R}, m(\xi)dxi)}.
\]
Hence the reproducing kernel of the strip space
\(\Hh\) is
\begin{equation}\label{eq:strip-kernel-fourier}
 K_{\Hh}(x,0)
 =
 \frac1{2\pi}
 \int_\R\frac{e^{ix\xi}}{m(\xi)}\,d\xi.
\end{equation}

On the other hand, by \eqref{eq:kernel-transform} and
\eqref{eq:disk-kernel}, since \(\psi(0)=0\),
\[
 K_{\Hh}(x,0)
 =
 \frac{\alpha+1}{4\pi}
 \sech^{\alpha+2}(x/2).
\]
Fourier inversion therefore gives
\[
 \frac1{m(\xi)}
 =
 \frac{\alpha+1}{4\pi}
 \int_\R
 e^{-ix\xi}\sech^{\alpha+2}(x/2)\,dx.
\]

With \(s=e^x\),
\[
\begin{aligned}
 \int_\R
 e^{-ix\xi}\sech^{\alpha+2}(x/2)\,dx
 &=
 2^{\alpha+2}
 \int_0^\infty
 \frac{
 s^{(\alpha+2)/2-i\xi-1}
 }{(1+s)^{\alpha+2}}
 \,ds\\
 &=
 2^{\alpha+2}
 \frac{
 \Gamma\!\left(\frac{\alpha+2}{2}-i\xi\right)
 \Gamma\!\left(\frac{\alpha+2}{2}+i\xi\right)
 }{\Gamma(\alpha+2)},
\end{aligned}
\]
by Euler's beta integral and the beta--gamma identity
\cite[Equations~6.2.1 and~6.2.2]{AbramowitzStegun}. 
Since
\[
 \Gamma(\alpha+2)
 =
 (\alpha+1)\Gamma(\alpha+1),
\]
we obtain \eqref{eq:gamma} for real \(\xi\).

Finally, both sides of \eqref{eq:gamma} are entire functions of \(z\),
thus the identity theorem gives \eqref{eq:gamma} for every \(z\in\C\).
\end{proof}

\begin{proposition}\label{lem:bifurcation}\label{lem:simple}
There exist \(\tau>0\) and \(0<B<3/2\) such that \(D_\tau\) has 
finitely many zeros in $
 0< \operatorname{Im}z\le B$, 
all of which are simple and have nonzero real part.  Moreover, at
least one of these zeros lies in $
 0<\operatorname{Im}z<B$.  
Let
 \[
s_0
:=
\min\left\{
\operatorname{Im}z:
D_\tau(z)=0,\ 
0\leq\operatorname{Im}z\le B
\right\}.
\] 
Then \(0<s_0<B\).   
\end{proposition}

\begin{proof}

\medskip
\noindent\emph{Step 1: Construction of a multiple zero on the imaginary axis.}

For \(0\le y\le3/2\), formula \eqref{eq:gamma} gives
\[
m(iy)
=
\frac{\pi\Gamma(\alpha+1)}
{2^\alpha
 \Gamma\!\left(\frac{\alpha+2}{2}-y\right)
 \Gamma\!\left(\frac{\alpha+2}{2}+y\right)}.
\]
Since \(\alpha>1\),
both
\(
 (\alpha+2)/{2}-y\) and \( (\alpha+2)/{2}+y
\)
are positive for \(0\le y\le3/2\). Hence
\[
 m(iy)>0,
 \qquad
 0\le y\le\frac32.
\]

For \(1/2<y<3/2\), define
\[
 T_\alpha(y)
 =
 -\frac{m(iy)}{2\cos(\pi y)}.
\]
Since \(\cos(\pi y)<0\) on this interval,
\(
 T_\alpha(y)>0.
\)
Moreover,
\[
 T_\alpha(y)\longrightarrow+\infty
 \qquad
 \text{as }y\downarrow\frac12
 \text{ or }y\uparrow\frac32,
\]
since \(m(iy)\) remains positive and finite at both endpoints, whereas
\(\cos(\pi y)\to0\).  Therefore \(T_\alpha\) attains its minimum at some
\(y_*\in(1/2,3/2)\).  Set 
\[
 \tau_*=T_\alpha(y_*)>0.
\]
For \(1/2<y<3/2\),
\begin{equation}\label{eq:Dtau-imaginary}
 D_\tau(iy)
 =
 m(iy)+2\tau\cos(\pi y)
 =
 2\cos(\pi y)\bigl(\tau-T_\alpha(y)\bigr).
\end{equation}
Thus $
 D_{\tau_*}(iy_*)=0$. 
Since \(y_*\) is an interior minimum of \(T_\alpha\),
\[
 T_\alpha'(y_*)=0.
\]
Differentiating \eqref{eq:Dtau-imaginary} with respect to \(y\) and
evaluating at \(y=y_*\) gives
\[
 \frac{d}{dy}D_{\tau_*}(iy)\bigg|_{y=y_*}=0.
\]
Since $
 \frac{d}{dy}D_{\tau_*}(iy)
 =
 iD_{\tau_*}'(iy)$, 
we obtain $ D_{\tau_*}'(iy_*)=0$. 
Hence \(iy_*\) is a zero of \(D_{\tau_*}\) of multiplicity at least two.

\medskip
\noindent\emph{Step 2: Perturbation away from the imaginary axis.}

Choose a closed disk \(\overline\Delta\) centered at \(iy_*\), contained
in the strip $\{
0<\operatorname{Im}z<\frac32\}$, 
and sufficiently small that \(D_{\tau_*}\) has no zero on
\(\partial\Delta\). Since
\[
D_\tau(z)-D_{\tau_*}(z)
=
2(\tau-\tau_*)\cosh(\pi z),
\]
and \(D_{\tau_*}\) has no zero on \(\partial\Delta\), we may choose
\(\varepsilon>0\), with \(\varepsilon<\tau_*/2\), such that
\[
2\varepsilon
\max_{z\in\partial\Delta}|\cosh(\pi z)|
<
\min_{z\in\partial\Delta}|D_{\tau_*}(z)|.
\]
Hence, whenever $
|\tau-\tau_*|\le\varepsilon$, 
we have
\[
|D_\tau(z)-D_{\tau_*}(z)|
<
|D_{\tau_*}(z)|,
\qquad z\in\partial\Delta.
\]
By Rouch\'e's theorem, \(D_\tau\) and \(D_{\tau_*}\) therefore have
the same number of zeros in \(\Delta\), counted with multiplicity.
In particular, \(D_\tau\) has at least one zero in \(\Delta\) for every $
\tau\in[\tau_*-\varepsilon,\tau_*+\varepsilon]$.

We now consider $
\tau\in (\tau_*-\varepsilon,\tau_*)$. 
Since \(\varepsilon<\tau_*/2\), all such \(\tau\) are positive.
If \(0<y\leq 1/2\), then \(\cos(\pi y)\geq 0\), and hence
\[
D_\tau(iy)
=
m(iy)+2\tau\cos(\pi y)>0.
\]
If \(1/2<y<3/2\), then
\[
\cos(\pi y)<0,
\qquad
T_\alpha(y)\ge\tau_*>\tau.
\]
Therefore, by \eqref{eq:Dtau-imaginary},
\[
D_\tau(iy)
=
2\cos(\pi y)\bigl(\tau-T_\alpha(y)\bigr)>0.
\]
Thus
\[
D_\tau(iy)\ne0,
\qquad
0<y<\frac32,
\]
for every $
\tau\in (\tau_*-\varepsilon,\tau_*)$.

Consequently, for every such \(\tau\), \(D_\tau\) has at least one zero
in \(\Delta\), while none of its zeros in \(\Delta\) lies on the
imaginary axis.

Choose \(B<3/2\) so that $\overline\Delta
\subset
\{z:0<\operatorname{Im}z<B\}$. 
We next show that the zeros of \(D_\tau\) in $\{
0\le \operatorname{Im}z\le B\}$ 
remain in a fixed bounded region, uniformly for $
\tau\in[\tau_*-\varepsilon,\tau_*+\varepsilon]$.

By \eqref{eq:complex-bound} and
\eqref{eq:laplace}, uniformly for \(0\le y\le B\),
\[
|m(x+iy)|
\le m(x)
=
O\bigl(e^{\pi|x|}|x|^{-\alpha-1}\bigr),
\qquad |x|\to\infty.
\]
On the other hand,
\[
|\cosh(\pi(x+iy))|^2
=
\sinh^2(\pi x)+\cos^2(\pi y)
\ge \sinh^2(\pi x),
\]
and hence
\[
|\cosh(\pi(x+iy))|
\gtrsim e^{\pi|x|},
\qquad |x|\to\infty,
\]
uniformly for \(0\le y\le B\).

Since
\(
\tau >0
\)
for
\(\tau\in[\tau_*-\varepsilon,\tau_*+\varepsilon]\), it follows that
\begin{equation}\label{eq:D-growth}
D_\tau(x+iy)
=
2\tau\cosh(\pi(x+iy))
\bigl(1+O(|x|^{-\alpha-1})\bigr),
\qquad |x|\to\infty,
\end{equation}
uniformly for
\[
0\le y\le B,
\qquad
\tau\in[\tau_*-\varepsilon,\tau_*+\varepsilon].
\]
Consequently, there exists \(R>0\), all zeros of \(D_\tau\) in $\{
0\le\operatorname{Im}z\le B\}$ 
lie in the fixed compact rectangle $
\{z:|\operatorname{Re}z|\le R,\ 0\le\operatorname{Im}z\le B\}$, 
uniformly for $
\tau\in[\tau_*-\varepsilon,\tau_*+\varepsilon]$.

\medskip
\noindent\emph{Step 3: Choice of \(\tau\) for which all zeros in the strip are simple.}

Suppose that \(z\) is a multiple zero of \(D_\tau\). Then
\[
D_\tau(z)=0,
\qquad
D_\tau'(z)=0.
\]
The only zero of \(\cosh(\pi \xi)\) in  the strip
\(\{
0\le\operatorname{Im}\xi\le B<\frac32
\}\)
is \(i/2\). Since $
D_\tau(i/2)=m(i/2)>0$, we have \(\cosh(\pi z)\ne0\).
Hence these two
equations give 
\begin{equation}\label{eq:tao}
\tau
=
-\frac{m(z)}{2\cosh(\pi z)}
\end{equation}
and
\begin{equation}\label{eq:multiple-zero-condition}
m'(z)\cosh(\pi z)
-
\pi m(z)\sinh(\pi z)
=
0.
\end{equation}
The entire function on the left-hand side of
\eqref{eq:multiple-zero-condition} is not identically zero.
Indeed, otherwise
\[
\left(
\frac{m(\xi)}{\cosh(\pi \xi)}
\right)'=0
\]
for all $\xi \in \mathbb{C}$ with $\cosh(\pi \xi)\neq 0$.
In particular,
\(m(x)/\cosh(\pi x)\) would be constant on the real axis.  However,
\eqref{eq:laplace} gives
\[
\frac{m(x)}{\cosh(\pi x)}
=
O(x^{-\alpha-1})
\longrightarrow0
\qquad(x\to+\infty),
\]
whereas
\[
\frac{m(0)}{\cosh0}=m(0)>0,
\]
a contradiction.

Hence the zeros of the entire function in
\eqref{eq:multiple-zero-condition} are discrete, and therefore equation \eqref{eq:multiple-zero-condition} has finitely many solutions in the fixed compact rectangle
obtained in Step~2.  
Consequently, \eqref{eq:tao} implies that only finitely many values of $
\tau\in[\tau_*-\varepsilon,\tau_*+\varepsilon]$ 
can produce a multiple zero of \(D_\tau\) in $\{
0\le\operatorname{Im}z\le B\}$.

We now choose $
\tau\in(\tau_*-\varepsilon,\tau_*)$ 
outside this finite set of exceptional values.  For this choice of
\(\tau\), Step~2 shows that \(D_\tau\) has at least one zero in
\(\Delta\), and hence at least one zero in $\{
0<\operatorname{Im}z<B\}$. 
Moreover,
\[
D_\tau(iy)\ne0,
\qquad
0<y<\frac32,
\]
so none of the zeros of \(D_\tau\) in $
0<\operatorname{Im}z\le B$ 
lies on the imaginary axis.  By the choice of \(\tau\), every zero of
\(D_\tau\) in $
0\le\operatorname{Im}z\le B$
 is simple.

By Step~2, all zeros of \(D_\tau\) in this closed strip lie in a fixed
compact rectangle.  Since \(D_\tau\) is not identically zero, it has finitely many zeros in the strip.  Finally,
\[
D_\tau(x)
=
m(x)+2\tau\cosh(\pi x)>0,
\qquad x\in\R,
\]
so there are no zeros on the real axis.

Thus, setting
\[
s_0
=
\min\left\{
\operatorname{Im}z:
D_\tau(z)=0,\ 
0\le\operatorname{Im}z\le B
\right\},\]
 we have $
0<s_0<B$.
\end{proof}

\subsection{Sign change of the limiting kernel}

We first record an elementary fact about finite trigonometric sums.

\begin{lemma}\label{lem:trig}
Suppose
\[
P(x)=\sum_{j=1}^N c_j e^{ir_jx} \quad \text{for}\quad x\in \mathbb{R}
\]
is real-valued and not identically zero, where the \(r_j\) are
distinct nonzero real numbers. Then
\[
\liminf_{x\to\infty}P(x)<0.
\]
\end{lemma}

\begin{proof}
Since the \(r_j\) are nonzero and distinct,
\[
\lim_{T\to\infty}\frac1T\int_0^T P(x)\,dx=0.
\]
Moreover, since \(P\) is real-valued,
\[
\lim_{T\to\infty}\frac1T\int_0^T P(x)^2\,dx
=
\sum_{j=1}^N |c_j|^2.
\]
Set
\[
S=\sum_{j=1}^N |c_j|^2>0,
\qquad
\delta=\frac{S}{2\|P\|_\infty}.
\]
Suppose, to the contrary, that
\[
P(x)\ge -\delta
\]
for all sufficiently large \(x\). Then, for such \(x\),
\[
0\le P(x)+\delta\le \|P\|_\infty+\delta,
\]
and hence
\[
(P(x)+\delta)^2
\le
(\|P\|_\infty+\delta)(P(x)+\delta).
\]
Taking averages over \([0,T]\) and letting \(T\to\infty\), the
contribution from the fixed initial interval is negligible, and we obtain
\[
S+\delta^2
\le
(\|P\|_\infty+\delta)\delta
=
\frac{S}{2}+\delta^2,
\]
a contradiction. Thus, 
\[
\liminf_{x\to\infty}P(x)\le -\delta<0.
\]
\end{proof}

\begin{proposition}\label{prop:negative}
For every $\alpha>1$, there exist $\tau>0$ and $x_*>0$ such that
$K_\tau(x_*)<0$. 
\end{proposition}
\begin{proof}
Fix the parameter $\tau$ supplied by Proposition~\ref{lem:simple} and write
$D=D_\tau$.  Let $s_0$ be the smallest positive imaginary part of zeros of
$D$ in the strip from that proposition.  Enumerate the zeros at this height
as
\[
 z_j=r_j+is_0,
 \qquad 1\le j\le N.
\]
All of them are simple, the numbers $r_j$ are distinct, and $r_j\ne0$.
Since there are only finitely many zeros in
$0\le\operatorname{Im}z\le B$, we may choose $ s_0<b<B$ 
so close to $s_0$ that there is no zero of $D$ with
$s_0<\operatorname{Im}z\le b$.

Fix $x>0$.  Apply the residue theorem to
\[
 z\longmapsto\frac{e^{ixz}}{D(z)}
\]
on the positively oriented rectangle $\Gamma$ with vertices
$-R,R,R+ib,-R+ib$, where $R$ is so large that the rectangles contain all of the simple poles at $z_1,\ldots,z_N$.
Thus, the residue theorem implies
\begin{equation}\label{residue}
  \frac{1}{2\pi}\int_{\Gamma}\frac{e^{ixz}}{D(z)}\,dz= i\sum_{j=1}^{N}\Res\left(\frac{e^{ixz}}{D(z)},z_j\right) = i\sum_{j=1}^{N}\frac{e^{ixz_j}}{D'(z_j)}.
\end{equation}

 On the right vertical side of  $\Gamma$, write
\[
 z=R+iy,\qquad 0\le y\le b.
\]
By \eqref{eq:D-growth}, there exist \(C>0\) and \(R_0>0\) such that
\[
 \frac1{|D(R+iy)|}\le Ce^{-\pi R},
 \qquad R\ge R_0,\quad 0\le y\le b.
\]
Therefore
\[
 \begin{aligned}
 \left|
 \int_0^b
 \frac{e^{ix(R+iy)}}{D(R+iy)}\,i\,dy
 \right|
 &\le
 \int_0^b
 \frac{e^{-xy}}{|D(R+iy)|}\,dy\\
 &\le
 Ce^{-\pi R}\int_0^b e^{-xy}\,dy\\
 &\le
 Cb\,e^{-\pi R}
 \longrightarrow0
 \end{aligned}
\]
as $R\to \infty$.

The same estimate holds on the left vertical side.  The integral over the left
vertical side also tends to zero as \(R\to\infty\).

On the upper horizontal line of  $\Gamma$, 
\eqref{eq:D-growth} gives
\[
 \frac1{D(\xi+ib)}=O(e^{-\pi|\xi|})
 \qquad(|\xi|\to\infty).
\]
Thus the upper integral is absolutely convergent on $\{z: \operatorname{Im}z=b\}$. Passing to the limit as $R\to\infty$, \eqref{residue} implies that
\[
 K_\tau(x)
 =i\sum_{j=1}^N\frac{e^{ixz_j}}{D'(z_j)}
 +\frac1{2\pi}\int_\R
   \frac{e^{ix(\xi+ib)}}{D(\xi+ib)}\,d\xi.
\]
The last integral satisfies
\[
 \left|
 \frac1{2\pi}\int_\R
   \frac{e^{ix(\xi+ib)}}{D(\xi+ib)}\,d\xi
 \right|
 \le Ce^{-bx},
\]
where
\[
 C=\frac1{2\pi}\int_\R\frac{d\xi}{|D(\xi+ib)|}<\infty.
\]
Since $z_j=r_j+is_0$, this gives the asymptotic formula
\begin{equation}\label{eq:residue}
 K_\tau(x)=e^{-s_0x}P(x)+O(e^{-bx}),
 \qquad
 P(x)=i\sum_{j=1}^N\frac{e^{ir_jx}}{D'(z_j)}.
\end{equation}

It remains to understand the sign of the trigonometric polynomial $P$.  Since $m$ is real on the real axis and is even,
\[
 D(-z)=D(z),
 \qquad
 D(\overline z)=\overline{D(z)}.
\]
Differentiating and combining these identities gives $
 D'(-\overline z)=-\overline{D'(z)}$. 
Thus, whenever $z=r+is_0$ is a zero, so is
$-\overline z=-r+is_0$, and the corresponding two terms in $P$ are complex
conjugates:
\[
 \frac{i e^{irx}}{D'(r+is_0)}
 \quad\text{and}\quad
 \frac{i e^{-irx}}{D'(-r+is_0)}
 =\overline{\frac{i e^{irx}}{D'(r+is_0)}}.
\]
Hence $P(x)$ is real-valued. Clearly, $P$ is not identically zero.

Since the \(r_j\) are
distinct nonzero real numbers, by Lemma~\ref{lem:trig}, there exist \(\delta>0\) and a sequence
\(x_k\to\infty\) such that
$
P(x_k)<-\delta.
$
Evaluating \eqref{eq:residue} along this sequence gives
\[
e^{s_0x_k}K_\tau(x_k)
=
P(x_k)+O(e^{-(b-s_0)x_k}).
\]
Since \(b>s_0\), the error tends to zero. Hence, for all sufficiently
large \(k\),
\[
e^{s_0x_k}K_\tau(x_k)<-\frac{\delta}{2}<0,
\]
and the
conclusion follows.
\end{proof}

\begin{remark}
The restriction \(\alpha>1\) in Proposition~\ref{lem:bifurcation} is
essential for the argument above.  At the endpoint \(\alpha=1\),
formula \eqref{eq:gamma} becomes $ m(z)=\frac{2\cosh(\pi z)}{1+4z^2}$. 
Hence, for \(1/2<y<3/2\),
\[
 T_1(y)
 =
 -\frac{m(iy)}{2\cos(\pi y)}
 =
 \frac{1}{4y^2-1}.
\]
In particular,
\[
 T_1(y)\longrightarrow\frac18
 \qquad\text{as }y\uparrow\frac32,
\]
rather than tending to \(+\infty\).  Thus the interior-minimum argument
used in Step~1 does not apply when \(\alpha=1\). 
\end{remark}

\begin{corollary}\label{cor:negative-discrete}
Suppose \(\tau>0\) and \(x_*>0\) be chosen as in
Proposition~\ref{prop:negative}.  Then $ K_h(x_*,0)<0$ 
for all sufficiently small \(h>0\).
\end{corollary}

\begin{proof}
Suppose \(\tau>0\) and \(x_*>0\) be chosen as in
Proposition~\ref{prop:negative}.
Since, by Theorem~\ref{thm:kernel-limit},
\[
K_h(x_*,0)\longrightarrow K_\tau(x_*)
\qquad (h\downarrow0),
\]
it follows that
\(
K_h(x_*,0)<0
\)
for all sufficiently small \(h>0\).
\end{proof}
\subsection{An infinite zero-based counterexample}

\begin{proposition}\label{prop:infinite-counterexample}
For every \(\alpha>1\), there exists an infinite set
\(A\subset\D\), invariant under complex conjugation and
containing no real points, such that \(I_{A}\) fails the
wandering subspace property.
\end{proposition}

\begin{proof}
Fix \(\tau>0\) and \(x_*>0\) as in
Proposition~\ref{prop:negative}. By
Corollary~\ref{cor:negative-discrete}, choose \(h>0\) so small that
\[
K_h(x_*,0)<0.
\]
Set
$
A=\psi(\Lambda_h).
$
Since \(\psi\) preserves complex conjugation
and maps the real axis onto \((-1,1)\),
the set \(A\) is invariant under
complex conjugation and contains no real points.
Moreover, \(I_{A}\neq\{0\}\), since
\(\mathcal M_h\neq\{0\}\).
Thus, Proposition~\ref{lem:real}(iii) and
Lemma~\ref{lem:extra} imply that
\(K_{A}(x,0)\) is real for
\(x\in(-1,1)\) and \(K_{A}(0,0)>0\).

Since \(J_\alpha\) is zero-free on \(S\),
Proposition~\ref{lem:unitary} gives $
 U^{-1}\mathcal M_h=I_{A}$. The kernel transformation
formula gives
\[
 K_h(x_*,0)
 =
 J_\alpha(x_*)J_\alpha(0)
 K_{A}(\psi(x_*),0).
\]
As \(J_\alpha>0\) on the real axis, we obtain
\[
 K_{A}(\psi(x_*),0)<0.
\]
By the Intermediate Value Theorem, there exists some \(c\in (0, \psi(x_*))\) such that  $
 K_{A}(c,0)=0$.
Lemma~\ref{lem:extra} therefore shows that \(I_{A}\) fails the wandering
subspace property.
\end{proof}

\begin{proof}[Proof of Corollary~\ref{cor:critical-range}]
For \(-1<\alpha\le1\), the assertion follows from
\cite[Theorem~1.4]{Shimorin}. For \(\alpha>1\), Proposition~\ref{prop:infinite-counterexample}
shows that the Beurling-type theorem fails.
\end{proof}

\section{Finite truncation and proof of the main theorem}
\label{sec:finite}
\label{app:limits}

Proposition~\ref{prop:infinite-counterexample} already shows that the
Beurling-type theorem fails for every \(\alpha>1\). The purpose
of this section is to strengthen this conclusion by replacing the
infinite periodic zero set by a finite one. This finite strengthening
is also needed for the application to the finite Blaschke-product
construction of Hedenmalm and Perdomo.

We shall use the following elementary projection identity for reproducing
kernels.  It is standard in RKHS theory and goes back to Aronszajn
\cite{Aronszajn}; see also \cite[Theorem~2.5]{PR}.

\begin{lemma}\label{lem:rkhs-projection}
Let \(H\) be a reproducing kernel Hilbert space with kernel \(K\), let
\(M\subset H\) be a closed subspace, and let \(P_M\) be the orthogonal projection onto
\(M\). Then
\begin{equation}\label{eq:rkhs-projection}
 K_M(\cdot,w)=P_MK(\cdot,w),
 \qquad
 K_M(z,w)=\langle P_MK(\cdot,w),K(\cdot,z)\rangle_H.
\end{equation}
Consequently,
\[
 K_M(w,w)\le K(w,w),
 \qquad
 |K_M(z,w)|\le \sqrt{K(z,z)K(w,w)}.
\]
\end{lemma}

We   record a standard fact about decreasing closed subspaces.

\begin{lemma}\label{lem:projections}
Let $ M_1\supset M_2\supset\cdots$ 
be closed subspaces of a Hilbert space \(H\), and let \(P_N\) denote
the orthogonal projection onto \(M_N\).  Put
\[
 M_\infty=\bigcap_{N\ge1}M_N.
\]
Then, \(
 P_N\to P_{M_\infty}\) in strong operator topology as
\((N\to\infty)\).

If \(H\) is an RKHS, then $
 K_{M_N}(z,w)\longrightarrow K_{M_\infty}(z,w)$ 
for every fixed \(z,w\).
\end{lemma}

\begin{proof}[Proof of Theorem~\ref{thm:main}]
By the proof of Proposition \ref{prop:infinite-counterexample},  set
$A$ is a subset of $\mathbb{D}$ which is invariant under complex
conjugation and contains no real numbers, and suppose $b_*\in (0,1)$ such that
\[
K_A(b_*,0)<0.
\]
For each positive number $n$, let
\[
A_n=\{a\in A: |a|\leq 1-1/n\}.
\]
Then $A_n$  is invariant under complex
conjugation and contains no real numbers. Consequently, by Lemma \ref{lem:extra},
\[
K_{A_n}(0,0)>0
\]
and
\(K_{A_n}(x,0)\) is real for $x\in (-1,1)$.

Since the sequence \(A_n\) is increasing and their union equals $A$, it follows that the corresponding zero-base invariant subspaces satisfy $I_{A_n}$ is decreasing and
\[
\cap_{n}I_{A_n}=I_A.
\]
Lemma \ref{lem:projections} yields
\(
K_{A_n}(b_*,0)\to K_{A}(b_*,0)
\), and hence there exists a positive number $N$ such that
\[
K_{A_N}(b_*,0)<0.
\]
By the Intermediate Value Theorem, there exists some \(c\in (0, b_*)\) such that  $
 K_{A_N}(c,0)=0$.
Lemma~\ref{lem:extra} therefore implies that
\(I_{A_N}\) fails the wandering subspace property.
\end{proof}

\begin{proof}[Proof of Corollary~\ref{cor:HP}]
Fix \(\alpha>1\). By Theorem~\ref{thm:main}, there exist a finite set
\(A\subset\mathbb D\) and \(c\in\mathbb D\setminus A\) such that
\[
 K_A(c,0)=0.
\]
Let \(B_A\) be the finite Blaschke product with zero set \(A\), and set
\[
 \omega_{\alpha,A}(z)
 =
 (\alpha+1)(1-|z|^2)^\alpha |B_A(z)|^2.
\]
By \cite[Proposition~4.1]{HP},
\[
 K_A(z,w)
 =
 B_A(z)\overline{B_A(w)}
 K_{\omega_{\alpha,A}}(z,w).
\]
Since \(B_A(c)B_A(0)\ne0\),
\[
 K_{\omega_{\alpha,A}}(c,0)=0.
\]

We now invoke the regularization argument of
\cite[Section~4]{HP}. With all parameters \(\rho_k=2\), the weight
\(\omega_{\alpha,A}\), up to a positive constant,
is precisely the limiting weight considered there. Hedenmalm and
Perdomo approximate the point masses by positive \(C^\infty\)-functions,
introduce a slight radial dilation of the hyperbolic density, and then
pass to the limiting weight by their ``limit process argument.''

Suppose, to the contrary, that the optimization problem admits a
smooth positive solution for every smooth curvature form
\(\boldsymbol{\mu}=\mu\,dA\) satisfying
\[
\boldsymbol{\mu}
+\frac{\alpha}{2}\boldsymbol K_{\mathbb H}\le0.
\]
For each regularized smooth datum, the Bergman kernel representation of
the extremal metric implies that the corresponding weighted Bergman
kernel is zero-free. The limit process argument of
\cite[Section~4]{HP} would then imply that the reproducing kernel for
the limiting weight \(\omega_{\alpha,A}\) is also zero-free. This
contradicts
\[
K_{\omega_{\alpha,A}}(c,0)=0.
\]

Consequently, there exists a smooth curvature form
\(\boldsymbol{\mu}=\mu\,dA\) satisfying
\[
\boldsymbol{\mu}
+\frac{\alpha}{2}\boldsymbol K_{\mathbb H}\le0
\]
for which the associated optimization problem admits no smooth positive
solution.
\end{proof}

\subsection*{Conflict of interest}
The authors declare no conflicts of interest relevant to the content of
this article.

\subsection*{Data availability statement}
The arguments in this article are analytic. No research datasets are
required for the proofs.

\subsection*{AI disclosure}
During the preparation of this work, the authors used ChatGPT 5.6 Sol (OpenAI)
to assist with language refinement, exposition, \LaTeX{} preparation,
and the exploration and checking of selected local mathematical arguments.
In particular, in Section~4, AI assistance was used in the development
and checking of the argument constructing simple zeros of the function
$D_\tau$, as well as in the subsequent proof applying the residue theorem
to establish the sign change of the limiting kernel. The key ideas underlying the remaining proofs were independently developed by the authors.
All mathematical statements, proofs, and conclusions were independently
checked by the authors, who take full responsibility for the content
of the paper.

\bibliographystyle{amsplain}
\bibliography{references}

\medskip

\noindent
School of Mathematical Sciences, Dalian University of Technology,
Dalian, Liaoning 116024, P. R. China

\noindent
Email address: \texttt{linzhaopeng2606@163.com} (Zhaopeng Lin)

\noindent
School of Mathematical Sciences, Dalian University of Technology,
Dalian, Liaoning 116024, P. R. China

\noindent
Email address: \texttt{shiboxu98@163.com} (Shibo Xu)

\noindent
School of Mathematical Sciences, Dalian University of Technology,
Dalian, Liaoning 116024, P. R. China

\noindent
Email address: \texttt{tyu@dlut.edu.cn} (Tao Yu)

\end{document}